\documentclass{amsart}
\usepackage{amsmath,amssymb,amsthm,enumerate}
\newtheorem{definition}{Definition}[section]
\newtheorem{proposition}[definition]{Proposition}
\newtheorem{lemma}[definition]{Lemma}
\newtheorem{theorem}[definition]{Theorem}
\newtheorem{corollary}[definition]{Corollary}
\newtheorem{remark}[definition]{Remark}

\title{Lie applicable surfaces}
\author{Mason Pember}
\address{Vienna University of Technology, Wiedner Hauptstrasse 8-10/104, A-1040 Vienna, Austria\\}
\email{mason@geometrie.tuwien.ac.at}

\begin{document}
\maketitle

\begin{abstract}
We give a detailed account of the gauge-theoretic approach to Lie applicable surfaces and the resulting transformation theory. In particular, we show that this approach coincides with the classical notion of $\Omega$- and $\Omega_{0}$-surfaces of Demoulin. 
\end{abstract}

\section{Introduction}

In~\cite[Section 85]{B1929}, Blaschke studies surfaces in Lie sphere geometry using the hexaspherical coordinate model introduced by Lie~\cite{L1872}. By using an adapted frame, Blaschke studies the compatibility conditions of such surfaces and in so doing finds that there are two 1-forms $\omega_{1}$ and $\omega_{2}$ that generically determine a surface up to Lie sphere transformation (one can alternatively use the quadratic form $\omega_{1}\omega_{2}$
and the conformal class of the cubic form $\omega_{1}^{3} - \omega_{2}^{3}$). Blaschke showed that there exist surfaces that are not determined by these forms. Following the terminology of~\cite{MN2006} we shall call these~\textit{Lie applicable surfaces}. In~\cite{MN2006} it is also shown that these surfaces are the deformable surfaces of Lie sphere geometry, that is, the only surfaces in Lie sphere geometry that admit non-trivial second order deformations. 

The class of Lie applicable surfaces consists only of $\Omega$- and $\Omega_{0}$-surfaces, the theory of which we shall now recall. Originally discovered by Demoulin~\cite{D1911iii, D1911i,D1911ii},  $\Omega$-surfaces in $\mathbb{R}^{3}$ are characterised (using standard notation) by the equation
\begin{equation}
\label{eqn:introdem} 
\left(\frac{V}{U}\frac{\sqrt{E}}{\sqrt{G}}\frac{\kappa_{1,u}}{\kappa_{1} - \kappa_{2}}\right)_{v} + \epsilon^{2}  \left(\frac{U}{V}\frac{\sqrt{G}}{\sqrt{E}}\frac{\kappa_{2,v}}{\kappa_{1}-\kappa_{2}}\right)_{u}=0
\end{equation}
given in terms of curvature line coordinates $(u,v)$, where $U$ is a function of $u$, $V$ is a function of $v$ and $\epsilon\in\{1,i\}$. Demoulin showed that $\Omega$-surfaces envelop a pair of isothermic sphere congruences and gave an alternative characterisation in terms of the existence of an associate $\Omega$-surface, analogous to the Christoffel transformation of isothermic surfaces. Furthermore, it is shown that isothermic, Guichard and $L$-isothermic surfaces are examples of $\Omega$-surfaces. Eisenhart~\cite{E1915,E1916} later developed a B\"{a}cklund-type transformation for these surfaces. $\Omega_{0}$-surfaces, the Lie geometric analogue of $R_{0}$-surfaces, are the surfaces satisfying~(\ref{eqn:introdem}) with $\epsilon=0$ and are envelopes of a curvature sphere congruence that is isothermic. 

Recent interest in integrable systems has sparked a renewed interest in $\Omega$- and $\Omega_{0}$-surfaces~\cite{BHR2010,BHR2012,C2012i, F2000ii,F2002,MN2006}.
Since isothermic surfaces~\cite{BCwip,BC2010i,BDPP2011,H2003,BS2012,S2008}, Guichard surfaces~\cite{BC2010i, BCwip, H2003, KS2003} and $L$-isothermic surfaces~\cite{MN1997, MN2000, MN2016, RS2009, S2009} have all been shown to constitute integrable systems, it comes as no surprise that $\Omega$- and $\Omega_{0}$-surfaces constitute such systems as well. In~\cite[Chapter 4]{C2012i}, Clarke develops a gauge-theoretic approach for Lie applicable surfaces (and, more generally, $l$-applicable maps) analogous to the approach used for isothermic surfaces, that is, they are characterised by the existence of a certain 1-parameter family of flat connections. This approach lends itself well to the study of transformations of these surfaces:
\begin{itemize}
\item local trivialising gauge transformations of these connections give rise to a spectral deformation, 
\item parallel sections give rise to B\"{a}cklund-type transformations, and
\item analogues of the well known permutability theorems for transformations of isothermic surfaces~\cite{BDPP2011,H2003} hold for these transformations.
\end{itemize}
Furthermore, certain well known examples of Lie applicable surfaces (e.g., linear Weingarten surfaces, see~\cite{BHR2010,BHR2012}) can be characterised in terms of polynomial conserved quantities of this family of flat connections. 

The purpose of this paper is to give a detailed account of the gauge theoretic approach for Lie applicable surfaces, revisiting and elaborating further on the work of Clarke~\cite{C2012i}. Particular attention is given to making clear the equivalence of this approach and the classical definition of Demoulin~\cite{D1911iii, D1911i,D1911ii}. 

In Section~\ref{sec:pre} we recall the Lie sphere model of~\cite{L1872}. In this setting we study the Legendre lift of a front in a three dimensional space form. We recover the invariants of such a lift introduced by Blaschke~\cite{B1929} and recall the modern approach to Ribaucour transforms of~\cite{BH2006}. 

In Section~\ref{sec:lieapp} Lie applicable surfaces are studied from the gauge theoretic viewpoint, that is, by the existence of a non-trivial closed 1-form taking values in a certain vector bundle. Such an approach is less straightforward than in the case of isothermic surfaces as, given such a closed 1-form, we obtain a set of uncountably many such closed 1-forms. This ambiguity is dealt with by using the middle potential - a unique 1-form in this set with a certain geometric property. This is analogous to the potential used in~\cite[\S2.4.1]{C2012i} for the study of projectively applicable surfaces. We show that this approach yields the classical notion of $\Omega$- and $\Omega_{0}$-surfaces~\cite{D1911iii,D1911i,D1911ii} in space forms. 

In Section~\ref{sec:trafos} we recall from~\cite{C2012i} the transformation theory of Lie applicable surfaces.  In contrast to~\cite{C2012i}, we give some consideration to umbilics. For example, we see that the appearance of umbilics on Darboux transforms is attributed to the enveloping sphere congruence between the two surfaces coinciding with one of the isothermic sphere congruences. 

In Section~\ref{sec:asssurf} we recall the classical notion of associate $\Omega$-surfaces~\cite{D1911iii}, i.e., two Combescure transformations such that a certain relation between the principal curvatures of the two surfaces is satisfied. We show that such surfaces give rise to a system of $O$-surfaces, see~\cite{KS2003}. 

\textit{Acknowledgements.} This work is based on part of the author's doctoral thesis~\cite{P2015}. The author would like to thank his PhD supervisor, Francis Burstall, for excellent supervision and subsequent support. Furthermore, he expresses his gratitude to the University of Bath and Kobe University for providing enjoyable environments to carry out this research. He is also very thankful to Udo Hertrich-Jeromin and Wayne Rossman for invaluable guidance and feedback. Moreover, he gratefully acknowledges support from the Japan Society for the Promotion of Science and the Engineering and Physical Sciences Research Council. The author would also like to thank the referees for their useful input. 

\section{Preliminaries}
\label{sec:pre}

\subsection{Notation}
Let $\Sigma$ be a manifold and, as usual, let $T\Sigma$ denote the tangent bundle of $\Sigma$. For a vector bundle $E$ over $\Sigma$, $\Gamma E$ shall denote the space of smooth sections of $E$. 
Given a vector space $V$, we shall denote by $\underline{V}$ the trivial bundle $\Sigma\times V$. If $W$ is a vector subbundle of $\underline{V}$, we define $W^{(1)}$ to be the subset of $\underline{V}$ consisting of the images of sections of $W$ and derivatives of sections of $W$ with respect to the trivial connection on $\underline{V}$ and call $W^{(1)}$ the derived bundle of $W$. In general $W^{(1)}$ will not be a subbundle of $\underline{V}$, however, in many instances, we may assume that it is.  

Throughout this paper we shall be considering the pseudo-Euclidean space $\mathbb{R}^{4,2}$, i.e., a six dimensional vector space equipped with a non-degenerate symmetric bilinear form $(\,,\,)$ of signature $(4,2)$. Let $\mathcal{L}$ denote the lightcone of $\mathbb{R}^{4,2}$. The orthogonal group $\textrm{O}(4,2)$ acts transitively on $\mathcal{L}$. We shall denote by $\mathbb{P}(\mathcal{L})$ the projectivisation of $\mathcal{L}$, i.e., the set of null 1-dimensional subspaces of $\mathbb{R}^{4,2}$. 

We shall recall in Subsection~\ref{subsec:symbreak} that, under Lie's~\cite{L1872} correspondence, points in $\mathbb{P}(\mathcal{L})$ correspond to spheres in any three dimensional space form. Therefore given a manifold $\Sigma$ we have that any smooth map $s:\Sigma\to \mathbb{P}(\mathcal{L})$ corresponds to a sphere congruence in any space form. We shall thus refer to $s$ as a~\textit{sphere congruence}. Such a map can also be identified as a smooth rank 1 null subbundle of the trivial bundle $\underline{\mathbb{R}}^{4,2}$. 

\begin{remark}
It is well known that the exterior algebra $\wedge^{2}\mathbb{R}^{4,2}$ is isomorphic to the Lie algebra $\mathfrak{o}(4,2)$ of $\textrm{O}(4,2)$, i.e., the space of skew-symmetric endomorphisms of $\mathbb{R}^{4,2}$, via the isomorphism
\[ a\wedge b\mapsto (a\wedge b),\]
where for any $c\in\mathbb{R}^{4,2}$, 
\[ (a\wedge b)c = (a,c)b - (b,c)a.\]
We shall make use of this identification (without warning) throughout this paper. 
\end{remark}

Given a manifold $\Sigma$, if $\omega_{1},\omega_{2}\in\Omega^{1}(\underline{\mathbb{R}}^{4,2})$, that is $\omega_{1}$ and $\omega_{2}$ are 1-forms on $\Sigma$ with values in $\underline{\mathbb{R}}^{4,2}$, then we define $\omega_{1}\curlywedge \omega_{2}$ to be the $2$-form with values in $\wedge^{2}\underline{\mathbb{R}}^{4,2}$ defined by
\[ \omega_{1}\curlywedge \omega_{2}(X,Y) := \omega_{1}(X)\wedge \omega_{2}(Y) - \omega_{1}(Y)\wedge \omega_{2}(X),\]
for $X,Y\in\Gamma T\Sigma$. Notice that $\omega_{1}\curlywedge \omega_{2} = \omega_{2}\curlywedge \omega_{1}$.

\subsection{Legendre immersions}
The maximal isotropic subspaces that exist in $\mathbb{R}^{4,2}$ are 2-dimensional. Let $\mathcal{Z}$ denote the Grassmannian of such isotropic 2-dimensional subspaces. Of course, we can identify this space with the space of lines in the projective lightcone $\mathbb{P}(\mathcal{L})$. We shall recall in Subsection~\ref{subsec:symbreak} that under Lie's correspondence~\cite{L1872} such lines correspond to parabolic pencils of spheres. 

Suppose that $\Sigma$ is a 2-dimensional manifold and let $f:\Sigma\to \mathcal{Z}$ be a smooth map. We may view $f$ as a 2-dimensional subbundle of the trivial bundle $\underline{\mathbb{R}}^{4,2}$. Then we may define a tensor, analogous to the solder form defined in~\cite{BC2004,BR1990}, 
\[ \beta: T\Sigma \to Hom(f,f^{(1)}/f),\quad X\mapsto (\sigma \mapsto d_{X}\sigma \, \bmod\, f).\]
In accordance with~\cite[Theorem 4.3]{C2008} we have the following definition:

\begin{definition}
$f:\Sigma\to\mathcal{Z}$ is a Legendre immersion if $f^{(1)}=f^{\perp}$ and $\ker\beta =\{0\}$. 
\end{definition}

Note that $f^{\perp}/f$ is a rank 2 subbundle of $\underline{\mathbb{R}}^{4,2}/f$, inheriting a positive definite metric from $\mathbb{R}^{4,2}$.    

Using the terminology of~\cite{BH2006} we say that $f$ \textit{envelops} a sphere congruence $s:\Sigma\to \mathbb{P}(\mathcal{L})$ if for all $p\in\Sigma$, $s(p)\subset f(p)$, i.e., $s$ is a rank 1 subbundle of $f$. 

\begin{definition}
\label{def:curvsph}
Let $p\in\Sigma$. Then a 1-dimensional subspace $s(p)\le f(p)$ is a curvature sphere of $f$ at $p$ if there exists a non-zero subspace $T_{s(p)}\le T_{p}\Sigma$ such that $\beta(T_{s(p)})s(p) = 0$. We call the maximal such $T_{s(p)}$ the curvature space of $s(p)$. 
\end{definition}

It was shown in~\cite{P1985} that at each point $p$ there is either one or two curvature spheres. We say that $p$ is an \textit{umbilic point of $f$} if there is exactly one curvature sphere $s(p)$ at $p$ and in that case $T_{s(p)}=T_{p}\Sigma$. Away from umbilic points we have that the curvature spheres form two rank 1 subbundles $s_{1},s_{2}\le f$ with respective curvature subbundles $T_{1}=\bigcup_{p\in \Sigma}T_{s_{1}(p)}$ and $T_{2}=\bigcup_{p\in \Sigma}T_{s_{2}(p)}$. We then have that $f=s_{1}\oplus s_{2}$ and $T\Sigma = T_{1}\oplus T_{2}$. 

Suppose that $f$ is umbilic-free. Then for each curvature subbundle $T_{i}$ we may define a rank 3 subbundle $f_{i}\le f^{\perp}$ as the set of sections of $f$ and derivatives of sections of $f$ along $T_{i}$. 
One can check that given any non-zero section $\sigma\in \Gamma f$ such that $\langle\sigma\rangle\cap s_{i} = \{0\}$ we have that
\[ f_{i} = f\oplus d\sigma(T_{i}).\]
Furthermore, 
\[ f^{\perp}/f = f_{1}/f\oplus_{\perp} f_{2}/f,\]
and each $f_{i}/f$ inherits a positive definite metric from that of $\mathbb{R}^{4,2}$. 

Let $\sigma_{1}\in\Gamma s_{1}$ and $\sigma_{2}\in\Gamma s_{2}$ be lifts of the curvature sphere congruences and let $X\in\Gamma T_{1}$ and $Y\in \Gamma T_{2}$. Then from Definition~\ref{def:curvsph} it follows immediately that 
\[ d_{X}\sigma_{1},d_{Y}\sigma_{2}\in\Gamma f.\]
Let 
\[ S_{1}:= \left\langle \sigma_{1},d_{Y}\sigma_{1},d_{Y}d_{Y}\sigma_{1}\right\rangle \quad \text{and}\quad 
S_{2}:= \left\langle \sigma_{2},d_{X}\sigma_{2},d_{X}d_{X}\sigma_{2}\right\rangle.\]
It was shown in~\cite{B1929} that $S_{1}$ and $S_{2}$ are orthogonal rank 3 subbundles of $\underline{\mathbb{R}}^{4,2}$ and the restriction of the metric on $\mathbb{R}^{4,2}$ to each $S_{i}$ has signature $(2,1)$. Furthermore, $S_{1}$ and $S_{2}$ do not depend on choices and we have the following orthogonal splitting
\[ \underline{\mathbb{R}}^{4,2} = S_{1}\oplus_{\perp} S_{2}\] 
of the trivial bundle. We refer to this splitting as the \textit{Lie cyclide splitting of $\underline{\mathbb{R}}^{4,2}$} because it can be identified with the Lie cyclides of $f$, i.e., a congruence of Dupin cyclides that make ``the most contact'' with $f$ at each point.

This splitting now yields a splitting of the trivial connection $d$ on $\underline{\mathbb{R}}^{4,2}$:
\[ d = \mathcal{D} + \mathcal{N},\]
where $\mathcal{D}$ is the direct sum of the induced connections on $S_{1}$ and $S_{2}$ and 
\begin{equation}
\label{eqn:lcycN}
\mathcal{N} = d - \mathcal{D}\in \Omega^{1}((Hom(S_{1},S_{2})\oplus Hom(S_{2},S_{1}))\cap \mathfrak{o}(4,2)).
\end{equation}
Since $S_{1}$ and $S_{2}$ are orthogonal, we have that $\mathcal{D}$ is a metric connection on $\underline{\mathbb{R}}^{4,2}$ and $\mathcal{N}$ is a skew-symmetric endomorphism. Hence, $\mathcal{N}\in\Omega^{1}(S_{1}\wedge S_{2})$. 

\begin{lemma}
\label{lem:N}
$\mathcal{N}f\le\Omega^{1}(f)$ and $\mathcal{N}(T_{2})s_{1}=0=\mathcal{N}(T_{1})s_{2}$.
\end{lemma}
\begin{proof}
Suppose that $\sigma_{1}\in\Gamma s_{1}$. Then for any $Y\in\Gamma T_{2}$, $d_{Y}\sigma_{1}\in\Gamma S_{1}$ and thus $\mathcal{N}_{Y}\sigma_{1}=0$. Furthermore, since $s_{1}$ is a curvature sphere, $d_{X}\sigma_{1}\in\Gamma f$. Hence, $\mathcal{N}s_{1}\le \Omega^{1}(f)$. A similar argument can be used for $s_{2}$. 
\end{proof}

\subsection{Symmetry breaking}
\label{subsec:symbreak}
Suppose that $\mathfrak{q},\mathfrak{p}\in \mathbb{R}^{4,2}$ are non-zero vectors such that $\mathfrak{q}\perp \mathfrak{p}$ and $\mathfrak{p}$ is not null. Then we may define a quadric 
\[ \mathfrak{Q}^{3}:=\{y\in\mathcal{L}: \,(y,\mathfrak{q})=-1,\, (y,\mathfrak{p})=0\}. \]
One can show that~(see, for example,~\cite{H2003, S2008}) $\mathfrak{Q}^{3}$ is isometric to a three dimensional space form with constant sectional curvature $\kappa = -|\mathfrak{q}|^{2}$. Lie~\cite{L1872} showed that each $s\in \mathbb{P}(\mathcal{L})$ can be identified with an oriented\footnote{Unless $s\in \langle\mathfrak{p}\rangle^{\perp}$, there exists exactly one other point $\tilde{s}\in \mathbb{P}(\mathcal{L})$ such that $s^{\perp}\cap \mathfrak{Q}^{3} =\tilde{s}^{\perp}\cap \mathfrak{Q}^{3}$. Therefore, each sphere in $\mathfrak{Q}^{3}$ is represented by exactly two points in $\mathbb{P}(\mathcal{L})$ and this gives rise to a notion of orientation (see~\cite{C2008}, for example).} sphere in this space form, namely the sphere determined in $\mathfrak{Q}^{3}$ by the set of points
\[ s^{\perp}\cap \mathfrak{Q}^{3}. \]
Furthermore, in this correspondence, two spheres are in oriented contact with each other if and only if their representatives in $\mathbb{P}(\mathcal{L})$ are orthogonal. Thus, lines in $\mathbb{P}(\mathcal{L})$ correspond to parabolic pencils of spheres, i.e., 1-parameter families of mutually touching spheres. If $|\mathfrak{p}|^{2}=\pm 1$, then 
\[ \mathfrak{P}^{3}:=\{y\in\mathcal{L}:\, (y,\mathfrak{q})=0,\, (y,\mathfrak{p})=-1\}\]
can be identified with the space of hyperplanes (complete, totally geodesic hypersurfaces) in this space form. 

\textit{Lie sphere transformations} are the transformations of space forms that map spheres to spheres and preserve oriented contact. Conveniently, in this model these are represented by the orthogonal transformations of $\mathbb{R}^{4,2}$. In fact $\textrm{O}(4,2)$ is a double cover for the set of Lie sphere transformations. A modern account of this correspondence is given in~\cite{C2008}. 

Given a Legendre immersion $f:\Sigma\to \mathcal{Z}$, we generically obtain a \textit{space form projection} $\mathfrak{f} := f\cap \mathfrak{Q}^{3}$ and a \textit{tangent plane congruence} $\mathfrak{t} := f\cap \mathfrak{P}^{3}$. The condition that $f$ is a Legendre immersion, ensures that $\mathfrak{f}$ is a front, i.e., a smooth map into $\mathfrak{Q}^{3}$ admitting a unit normal vector such that the pairing of surface and normal is an immersion. Conversely, given a front $\mathfrak{f}:\Sigma\to \mathfrak{Q}^{3}$ with tangent plane congruence $\mathfrak{t}:\Sigma\to \mathfrak{P}^{3}$, we obtain a Legendre immersion by taking the span, $f= \langle \mathfrak{f},\mathfrak{t}\rangle$. 

Suppose that $\mathfrak{f}$ is an immersion. Then away from umbilic points of $\mathfrak{f}$ we may choose curvature line coordinates $(u,v)$. By Rodrigues' equations one has that 
\[\mathfrak{t}_{u}+\kappa_{1}\mathfrak{f}_{u} =0 = \mathfrak{t}_{v}+\kappa_{2}\mathfrak{f}_{v},\]
where $\kappa_{1}$ and $\kappa_{2}$ are the principal curvatures of $\mathfrak{f}$. Therefore, 
\[ s_{1}:= \langle \mathfrak{t} +\kappa_{1}\mathfrak{f}\rangle \quad \text{and}\quad s_{2} :=  \langle \mathfrak{t} +\kappa_{2}\mathfrak{f}\rangle\]
are curvature spheres of $f$ with respective curvature subbundles $T_{1}:= \left\langle \frac{\partial}{\partial u}\right\rangle$ and $T_{2}:= \left\langle \frac{\partial}{\partial v}\right\rangle$.

\subsection{Invariants of Lie sphere geometry}
We will now recover the Lie-invariant metric and conformal class of the cubic form used in~\cite{B1929, F2000ii}. These invariants generically\footnote{Blaschke~\cite{B1929} showed that those surfaces that aren't determined are the Lie applicable surfaces. We shall explore this further in Subsection~\ref{subsec:cal}.}  determine a surface up to Lie sphere transformation. 

Let $f:\Sigma\to\mathcal{Z}$ be a Legendre immersion. 

\subsubsection{Conformal structure}
\label{subsec:confstr}
Define a tensor $c\in \Gamma( S^{2}T^{*}\Sigma\otimes (\wedge^{2}f)^{*}\otimes \wedge^{2}(f^{\perp}/f))$ by 
\[ c(X,Y)\xi_{1}\wedge\xi_{2} = \frac{1}{2}(\beta(X)\xi_{1}\wedge \beta(Y)\xi_{2} + \beta(Y)\xi_{1}\wedge \beta(X)\xi_{2}),\]
for any $X,Y\in \Gamma T\Sigma$ and $\xi_{1},\xi_{2}\in\Gamma f$. Since the rank 2 bundle $f^{\perp}/f$ inherits a non-degenerate metric from $\mathbb{R}^{4,2}$, the rank 1 bundle $\wedge^{2}(f^{\perp}/f)$ inherits a definite metric and thus $\wedge^{2}(f^{\perp}/f)$ is a trivial bundle and we can view $c$ as a tensor in $S^{2}T^{*}\Sigma\otimes (\wedge^{2}f)^{*}$. Now suppose that $s(p)$ is a curvature sphere of $f$ at $p$ with curvature subspace $T_{s(p)}$. Then $\beta(T_{s(p)})s(p) = 0$ and since we may write any $\tau\in\Gamma (\wedge^{2}f)$ as $\tau =\sigma\wedge \tilde{\sigma}$, for some $\sigma,\tilde{\sigma}\in\Gamma f$ such that $\sigma(p)\in s(p)$, we have that
\[ c(T_{s(p)},T_{s(p)})\tau_{p}=0.\]
Hence, $c(T_{s(p)},T_{s(p)})=0$. Therefore, at umbilic points $p\in\Sigma$ of $f$, $c_{p} = 0$ and away from umbilic points, for any nowhere zero $\tau\in\Gamma(\wedge^{2}f)$, $g:=c\,\tau$ defines an indefinite metric on $\Sigma$ whose null lines are the curvature subbundles $T_{1}$ and $T_{2}$. We shall refer to $g$ as a representative metric of $c$ and, since $c$ is tensorial in $\wedge^{2}f$, we have that any other representative metric of $c$ is conformally equivalent to $g$. We shall thus refer to $c$ as \textit{the conformal structure of $f$}. 

In the case that $f$ is umbilic-free, the conformal structure $c$ gives rise to the Hodge star operator $\star$ which acts as $id$ on $T^{*}_{1}$ and $-id$ on $T^{*}_{2}$.

\subsubsection{Lie-invariant metric}
Now suppose that $f$ is an umbilic-free Legendre immersion. Recall from~(\ref{eqn:lcycN}) that the Lie cyclide splitting induces a skew-symmetric endomorphism $\mathcal{N}\in\Omega^{1}(S_{1}\wedge S_{2})$. By Lemma~\ref{lem:N}, $\mathcal{N}f\le \Omega^{1}(f)$. Therefore, we may define a tensor $g^{L}\in \Gamma( S^{2}T^{*}\Sigma\otimes \textrm{End}(\wedge^{2}f))$ called the \textit{Lie-invariant metric}\footnote{In~\cite{BH2002}, the Lie cyclides are shown to define a conformal Gauss map for $f$. One can show that the induced metric of this conformal Gauss map is a non-zero constant scalar multiple of the Lie invariant metric.} by 
\begin{eqnarray}
\label{eqn:liemetric} g^{L}(X,Y)\xi_{1}\wedge\xi_{2} = \frac{1}{2}(\mathcal{N}(X)\xi_{1}\wedge \mathcal{N}(Y)\xi_{2} + \mathcal{N}(Y)\xi_{1}\wedge \mathcal{N}(X)\xi_{2}),
\end{eqnarray}
for any $X,Y\in \Gamma T\Sigma$ and $\xi_{1},\xi_{2}\in\Gamma f$. Since $\wedge^{2}f$ has rank 1, $\textrm{End}(\wedge^{2}f)$ is canonically trivial and so we identify $g^{L}$ with a quadratic form. By Lemma~\ref{lem:N}, the curvature subbundles $T_{1}$ and $T_{2}$ are isotropic with respect to $g^{L}$ and thus, away from points where it vanishes, $g^{L}$ is a representative metric of $c$. 

\begin{remark}
Unlike the conformal structure $c$, $g^{L}$ may vanish at certain points. For example, if $f$ is a Dupin cyclide then $g^{L}\equiv 0$. 
\end{remark}

Recall that given a space form $\mathfrak{Q}^{3}$ and space form projection $\mathfrak{f}:\Sigma\to\mathfrak{Q}^{3}$ of $f$ with tangent plane congruence $\mathfrak{t}:\Sigma\to\mathfrak{P}^{3}$, we have that 
\[ \mathfrak{t}+\kappa_{1}\mathfrak{f} \quad \text{and}\quad \mathfrak{t}+\kappa_{2}\mathfrak{f}\]
are lifts of the curvature spheres $s_{1}$ and $s_{2}$, respectively. Now we may split the trivial connection $d=d_{1}+d_{2}$, where $d_{i}$ denotes the partial connection along $T_{i}$. Then one can check that 
\[ \mathcal{N}(\mathfrak{t}+\kappa_{1}\mathfrak{f}) = -\frac{d_{1}\kappa_{1}}{\kappa_{1}-\kappa_{2}}(\mathfrak{t}+\kappa_{2}\mathfrak{f}) \quad \text{and} \quad \mathcal{N}(\mathfrak{t}+\kappa_{2}\mathfrak{f}) = \frac{d_{2}\kappa_{2}}{\kappa_{1}-\kappa_{2}}(\mathfrak{t}+\kappa_{1}\mathfrak{f}).\]
Hence, in terms of curvature line coordinates $(u,v)$,
\[ g^{L} = (\kappa_{1}-\kappa_{2})^{-2}\kappa_{1,u}\kappa_{2,v}\, dudv,\]
and thus $g^{L}$ coincides with the Lie-invariant metric of~\cite[Theorem 1]{F2000ii}.

\subsubsection{Darboux cubic form}
\label{subsec:darbcubform}
Suppose that $f$ is an umbilic-free Legendre immersion. For $X,Y,Z\in\Gamma T\Sigma$ and $\xi_{1},\xi_{2}\in\Gamma f$, define a map 
\[ \mathcal{C}(X,Y,Z)\xi_{1}\wedge \xi_{2}:= (\mathcal{D}_{X}\mathcal{D}_{Y}\xi_{1},\mathcal{N}_{Z}\xi_{2}) -(\mathcal{D}_{X}\mathcal{D}_{Y}\xi_{2},\mathcal{N}_{Z}\xi_{1}) .\]
We call $\mathcal{C}$ the \textit{Darboux cubic form of $f$}.

\begin{lemma}
$\mathcal{C}$ is a tensor taking values in $((T_{1}^{*})^{3} \oplus (T_{2}^{*})^{3})\otimes (\wedge^{2}f)^{*}$. 
\end{lemma}
\begin{proof}
The tensorial nature of $\mathcal{C}$ follows from the fact that for any $X,Y,Z\in\Gamma T\Sigma$, $\xi\in\Gamma f$ and any smooth function $\lambda$,
\[
\mathcal{D}_{X}\mathcal{D}_{Y}(\lambda \xi)= \mathcal{D}_{X}\mathcal{D}_{\lambda Y} \xi=\mathcal{D}_{\lambda X}\mathcal{D}_{Y} \xi = \lambda\, \mathcal{D}_{X}\mathcal{D}_{Y}\xi\, \bmod\, f^{\perp}
\]
and by Lemma~\ref{lem:N}, $\mathcal{N}_{Z}f\le f$. 

Let $Z\in\Gamma T_{1}$, $\sigma_{1}\in\Gamma s_{1}$ and $\sigma_{2}\in\Gamma s_{2}$. Then by Lemma~\ref{lem:N}, $\mathcal{N}_{Z}\sigma_{2}=0$, and thus for any $X,Y\in\Gamma T\Sigma$,
\[ \mathcal{C}(X,Y,Z)\sigma_{1}\wedge \sigma_{2} = -(\mathcal{D}_{X}\mathcal{D}_{Y}\sigma_{2},\mathcal{N}_{Z}\sigma_{1}).\]
If either of $X$ or $Y$ lies in $T_{2}$ then $\mathcal{D}_{X}\mathcal{D}_{Y}\sigma_{2}\in\Gamma f^{\perp}$ and, since $\mathcal{N}_{Z}f\le f$, this would imply that $\mathcal{C}(X,Y,Z)=0$. A similar argument shows that if $Z\in\Gamma T_{2}$ and either of $X$ and $Y$ lies in $T_{1}$ then $\mathcal{C}(X,Y,Z)=0$. Hence, 
\[ \mathcal{C}\in\Gamma (((T_{1}^{*})^{3} \oplus (T_{2}^{*})^{3})\otimes (\wedge^{2}f)^{*})\]
as required. 
\end{proof}

\begin{remark}
\label{rem:cubspf}
By evaluating the Darboux cubic form $\mathcal{C}$ on $\tau:=(\mathfrak{t} +\kappa_{1}\mathfrak{f})\wedge(\mathfrak{t} +\kappa_{2}\mathfrak{f})$ one obtains
\[ \mathcal{C}\tau = 
(\kappa_{2}-\kappa_{1})(\kappa_{1,u}E\,du^{3} + \kappa_{2,v}G\, dv^{3}),\]
in terms of curvature line coordinates $(u,v)$. Hence, $\mathcal{C}\tau$ is in the same conformal class as the cubic form used in~\cite[Theorem 1]{F2000ii}.
\end{remark}

\subsection{Ribaucour transforms}
\label{sec:ribaucour}
In~\cite{BH2006}, a modern treatment of Ribaucour transforms was developed in the realm of Lie sphere geometry. In this section we shall recall this construction and prove some results that will be useful to us later in Subsection~\ref{subsec:darb} when considering Darboux transforms.

Suppose that $f,\hat{f}:\Sigma\to \mathcal{Z}$ are pointwise distinct Legendre immersions enveloping a common sphere congruence $s_{0}:=f\cap\hat{f}$. Then $s_{0}^{\perp}/s_{0}$ is a rank 4 subbundle of $\underline{\mathbb{R}}^{4,2}/s_{0}$ that inherits a non-degenerate metric with signature $(3,1)$ from $\mathbb{R}^{4,2}$. Let
\[ \mathcal{N}_{f,\hat{f}}:= (f+\hat{f})/s_{0}.\]
Then $\mathcal{N}_{f,\hat{f}}$ is a rank 2 subbundle of $s_{0}^{\perp}/s_{0}$ and the induced metric $\langle .,.\rangle$ on $\mathcal{N}_{f,\hat{f}}$ is non-degenerate with signature $(1,1)$. We then have a well-defined orthogonal projection $\pi:s_{0}^{\perp}/s_{0}\to\mathcal{N}_{f,\hat{f}}$. From the contact condition on $f$ and $\hat{f}$, one quickly deduces the following lemma:

\begin{lemma}
\label{lem:envperp} 
$s_{0}^{(1)}\le (f+\hat{f})^{\perp}$ and $(f+\hat{f})^{(1)}\le s_{0}^{\perp}$.
\end{lemma}
We now define a metric connection on $\mathcal{N}_{f,\hat{f}}$: for $\xi\in\Gamma(f+\hat{f})$,
\[ \nabla^{f,\hat{f}}(\xi + s_{0}) = \pi(d\xi + s_{0})\]
and make the following definition:

\begin{definition}
\label{def:rib}
If $\nabla^{f,\hat{f}}$ is flat then we say that $s_{0}$ is a Ribaucour sphere congruence and that $f$ and $\hat{f}$ are Ribaucour transforms of each other. 
\end{definition}

Now $f+\hat{f}$ is a rank 3 degenerate subbundle of $\underline{\mathbb{R}}^{4,2}$. If we let $l\le f+\hat{f}$ be a rank 2 subbundle of $f+\hat{f}$ such that $l\cap s_{0}=\{0\}$, then the induced metric on $l$ has signature $(1,1)$. This yields a splitting
\[ \underline{\mathbb{R}}^{4,2} = l\oplus l^{\perp}\]
and the trivial connection splits accordingly as 
\[ d = \mathcal{D}^{l} + \mathcal{D}^{l^{\perp}} +\mathcal{N}^{l,l^{\perp}},\]
where $\mathcal{D}^{l}$ is the induced connection on $l$, $\mathcal{D}^{l^{\perp}}$ is the induced connection on $l^{\perp}$ and 
\[\mathcal{N}^{l,l^{\perp}}= d - (\mathcal{D}^{l} + \mathcal{D}^{l^{\perp}})\in \Omega^{1}(Hom(l,l^{\perp})\oplus Hom(l^{\perp},l)).\]

\begin{proposition}
The vector bundle isomorphism 
\[ \psi:l\to \mathcal{N}_{f,\hat{f}}, \quad \xi\mapsto \xi+s_{0}\] 
preserves the metric and connection on $l$, i.e., $\psi^{*}\langle .,.\rangle = (.,.)|_{l\times l}$ and $\nabla^{f,\hat{f}}\circ \psi = \psi\circ \mathcal{D}^{l}$. 
\end{proposition}
\begin{proof}
Suppose that $\xi_{1},\xi_{2}\in\Gamma l$. Then 
\[ \langle \psi(\xi_{1}), \psi(\xi_{2})\rangle = \langle \xi_{1}+s_{0},  \xi_{2}+s_{0}\rangle = (\xi_{1},\xi_{2}).\]
Hence, the induced metric on $l$ is isometric to $\langle .,.\rangle$ via $\psi$. Furthermore, for $\xi\in \Gamma l$, 
\[ \nabla^{f,\hat{f}}(\psi(\xi)) = \pi(d\xi + s_{0}) = \mathcal{D}^{l}\xi + s_{0} = \psi(\mathcal{D}^{l}\xi).\]
Hence, $\psi$ is connection preserving. 
\end{proof}
This gives rise to an alternative characterisation of Ribaucour transforms:

\begin{corollary}
\label{cor:lflat}
$f$ and $\hat{f}$ are Ribaucour transforms of each other if and only if the induced connection $\mathcal{D}^{l}$ is flat for some (and hence all) $l\le f+\hat{f}$ of rank 2 such that $l\cap s_{0}=\{0\}$. 
\end{corollary}

\begin{remark}
Suppose that $l\cap s_{0}=\{0\}$ and let $s:=l\cap f$ and $\hat{s}:=l\cap\hat{f}$. Then the condition that $\mathcal{D}^{l}$ be flat is equivalent to requiring $s$ and $\hat{s}$ to be parallel subbundles of $\mathcal{D}^{l}$. In fact, $s$ being a parallel subbundle of $\mathcal{D}^{l}$ implies that $\hat{s}$ is parallel as well, and conversely.
\end{remark}

It was shown in~\cite{BH2006} that Definition~\ref{def:rib} is equivalent to the classical definition of Ribaucour transform~\cite{B1929,DT2002,DT2003,E1962,T2002}, that is, that the curvature directions of $f$ and $\hat{f}$ correspond. Suppose that $f$ and $\hat{f}$ are umbilic-free and let $s_{1},s_{2}\le f$ denote the curvature sphere congruences of $f$ and let $\hat{s}_{1},\hat{s}_{2}\le \hat{f}$ denote the curvature sphere congruences of $\hat{f}$. Then we may assume that $T_{i}$ is the curvature subbundle of $s_{i}$ and $\hat{s}_{i}$ for $i\in\{1,2\}$. Let 
\[ l_{i}:= s_{i}\oplus \hat{s}_{i}.\]
Then for any $\xi\in\Gamma l_{i}$ we have that $d\xi (T_{i})\le (f+\hat{f})$. Now let 
\[ s_{\infty} := l_{1}\cap l_{2}.\]
Then for any $\sigma_{\infty} \in \Gamma s_{\infty}$, we have that $d\sigma_{\infty}(T_{1})\le f+\hat{f}$, since $\sigma_{\infty}\in \Gamma l_{1}$ and $d\sigma_{\infty}(T_{2})\le f+\hat{f}$, since $\sigma_{\infty}\in \Gamma l_{2}$. Therefore, as $T\Sigma =T_{1}\oplus T_{2}$, $d\sigma_{\infty}\in\Omega^{1}(f+\hat{f})$. In fact $s_{\infty}$ is the unique point map in $\mathbb{P}(f+\hat{f})$ with the property that 
\[ s_{\infty}^{(1)}\le f+\hat{f}\]
and this motivates the following definition:

\begin{definition}
\label{def:envpoint}
We call $s_{\infty}$ the enveloping point of $f+\hat{f}$.
\end{definition}

\section{Lie applicable surfaces}
\label{sec:lieapp}
In this section we shall adopt the gauge theoretic viewpoint of Lie applicable surfaces laid out by Clarke~\cite{C2012i}. From this viewpoint, Lie applicability corresponds to the existence of a vector-bundle valued 1-form. The existence of such a 1-form gives rise to a set of uncountably many such 1-forms. In order to work with such a set, we geometrically derive a unique member called the middle potential. This is analogous to a potential used in~\cite[\S2.4.1]{C2012i} for studying projectively applicable surfaces.  

Given a Legendre immersion $f:\Sigma \to \mathcal{Z}$ we may consider the subbundle $f\wedge f^{\perp}$ of $\wedge^{2}\underline{\mathbb{R}}^{4,2}$. Now suppose that $\eta\in \Omega^{1}(f\wedge f^{\perp})$, i.e., $\eta$ is a 1-form taking values in $f\wedge f^{\perp}$. Then for any section $\sigma\in\Gamma f$, since $d\sigma\in \Omega^{1}(f^{\perp})$, we have that $\eta(X)d_{Y}\sigma \in \Gamma f$,
for any $X,Y\in \Gamma T\Sigma$. Furthermore, since $\eta f =0$, we have that $\eta(X)d_{Y}\sigma$ is tensorial in $\sigma$. Thus, for given $X,Y\in \Gamma T\Sigma$, we have an endomorphism on $f$ defined by 
\[ \sigma \mapsto \eta(X)d_{Y}\sigma.\]
Therefore we may take the trace of this endomorphism and this gives rise to a 2-tensor $q$ defined by 
\[ q(X,Y) = tr(\sigma \mapsto \eta(X)d_{Y}\sigma).\]
We are now in a position to state the main definition of this section:

\begin{definition}
\label{def:lieapp}
We say that $f$ is a Lie applicable surface if there exists a closed $\eta\in \Omega^{1}(f\wedge f^{\perp})$ such that $[\eta\wedge\eta]=0$ and $q$ is non-zero. 

Furthermore, if $q$ is non-degenerate (respectively, degenerate) on a dense open subset of $\Sigma$ we say that $f$ is an $\Omega$-surface ($\Omega_{0}$-surface). 
\end{definition}

Suppose now that $\eta\in \Omega^{1}(f\wedge f^{\perp})$ is closed. Then for any $\tau\in \Gamma(\wedge^{2}f)$, $\tilde{\eta}:= \eta - d\tau$ is a closed 1-form with values in $f\wedge f^{\perp}$. In this case we say that $\tilde{\eta}$ and $\eta$ are \textit{gauge equivalent}\footnote{In Section~\ref{sec:trafos} we shall see that each closed 1-form $\eta$ gives rise to a 1-parameter family of flat connections. Moreover, we shall see that if two 1-forms $\eta$ and $\tilde{\eta}$ are gauge equivalent then the resulting flat connections are related by a gauge transformation.}. This yields an equivalence relation on closed 1-forms with values in $f\wedge f^{\perp}$ and we call the equivalence class
\[ [\eta] := \{\eta - d\tau\, : \, \tau\in \Gamma (\wedge^{2}f)\}\]
the \textit{gauge orbit of $\eta$}. Of course, any Legendre immersion admits 1-forms with trivial gauge orbit, namely, $d\tau$ for any $\tau\in \Gamma(\wedge^{2}f)$. However, assuming that the quadratic differential $q$ is non-zero in Definition~\ref{def:lieapp} ensures that the associated 1-form is non-trivial:

\begin{lemma}
\label{lem:trivq}
If $[\eta]=[0]$ at $p\in \Sigma$ then $q=0$ at $p$.
\end{lemma}
\begin{proof}
Suppose that $s(p)$ is a curvature sphere congruence of $f$ at $p$ with associated curvature space $T_{s(p)}$. Then, for $\sigma\in \Gamma f$ and $X,Y\in \Gamma T\Sigma$ such that $X_{p},Y_{p}\in T_{s(p)}$, one has that $(d_{X}d_{Y}\sigma)_{p}\in s^{\perp}(p)$. Therefore, for any $\tau\in \Gamma (\wedge^{2}f)$, 
\[ (d_{X_{p}}\tau)d_{Y_{p}}\sigma = - \tau(p) (d_{X}d_{Y}\sigma)_{p} \in s(p).\]
Furthermore, if $\sigma(p)\in s(p)$ then $(d_{X}d_{Y}\sigma)_{p}\in f^{\perp}(p)$ and so $(d_{X_{p}}\tau)d_{Y_{p}}\sigma$ vanishes. Hence, $q=0$ at $p$. 
\end{proof}

\begin{corollary}
$q$ is well defined on gauge orbits, i.e., if $\tilde{\eta}\in [\eta]$ then $\tilde{q}=q$, where $\tilde{q}$ is the quadratic form associated to $\tilde{\eta}$.
\end{corollary}
\begin{proof}
This follows from the fact that $\tilde{\eta} - \eta = d\tau$ for some $\tau\in \Gamma(\wedge^{2}f)$. 
\end{proof}

Now suppose that $f$ is umbilic-free. That is, there are two distinct curvature sphere congruences $s_{1}$ and $s_{2}$ such that $s_{1}\cap s_{2}=\{0\}$. Let $T_{i}\le T\Sigma$ denote the corresponding rank 1 curvature subbundle for $s_{i}$, i.e., for any $X_{i}\in \Gamma T_{i}$ and $\sigma_{i}\in \Gamma s_{i}$, 
\[ d_{X_{i}}\sigma_{i}\in \Gamma f.\]
Recall that each curvature subbundle $T_{i}$ induces a rank 3 subbundle $f_{i}$ of $f^{\perp}$. The following proposition shows that in the umbilic-free case, we may drop the condition that $[\eta\wedge\eta]=0$ in Definition~\ref{def:lieapp}:

\begin{proposition}
\label{prop:mceqn}
$\eta$ is closed if and only if $\eta$ satisfies the Maurer Cartan equation. In this case, $\eta(T_{i})\le f\wedge f_{i}$ and $[\eta\wedge \eta]=0$. 
\end{proposition}
\begin{proof}
Since $\eta f \equiv 0$, we have that 
\[ (d\eta + \frac{1}{2}[\eta\wedge \eta])f =(d\eta) f.\]
Let $X_{i}\in \Gamma T_{i}$ and $X_{j}\in\Gamma T_{j}$ for $i\neq j$ and $\sigma_{i}\in \Gamma s_{i}$. Then
\begin{align*} 
d\eta(X_{i},X_{j})\sigma_{i} &= (d_{X_{i}}(\eta(X_{j})) - d_{X_{j}}(\eta(X_{i})) - \eta([X_{i},X_{j}]))\sigma_{i} \\
&= d_{X_{i}}(\eta(X_{j})\sigma_{i}) - \eta(X_{j})d_{X_{i}}\sigma_{i}  - d_{X_{j}}(\eta(X_{i})\sigma_{i}) + \eta(X_{i})d_{X_{j}}\sigma_{i}\\
&= - \eta(X_{j})d_{X_{i}}\sigma_{i} + \eta(X_{i})d_{X_{j}}\sigma_{i},
\end{align*}
using again that $\eta f\equiv 0$. Since $s_{i}$ is a curvature sphere, $d_{X_{i}}\sigma_{i}\in \Gamma f$ and thus $\eta(X_{j})d_{X_{i}}\sigma_{i}=0$. Therefore assuming that $\eta$ satisfies the Maurer-Cartan equation or that it is closed implies that for all $i\neq j$, $X_{i}\in\Gamma T_{i}$, $X_{j}\in \Gamma T_{j}$ and $\sigma_{i}\in \Gamma s_{i}$, 
\[ 0 = \eta(X_{i}) d_{X_{j}}\sigma_{i}.\]
Thus, $\eta(X_{i})\in \Gamma (f\wedge f_{i})$ and 
\[ [\eta(X_{i}), \eta(X_{j})]=0.\]
Thus, 
\[ [\eta\wedge \eta](X_{i},X_{j}) = 2[\eta(X_{i}), \eta(X_{j})]=0.\]
Therefore, since $X_{1}$ and $X_{2}$ form a basis for $T\Sigma$, we have that $[\eta\wedge \eta]=0$. Hence, 
\[ d\eta + \frac{1}{2}[\eta\wedge \eta] = d\eta \]
and the result follows.
\end{proof}

\begin{corollary}
\label{cor:qsym}
$q$ is symmetric with $q(T_{1}, T_{2})=0$. Hence, $q$ is a quadratic differential with respect to the conformal structure $c$. 
\end{corollary}

In order to work with the gauge orbit of closed 1-forms that arises from Lie applicability, we shall  derive a unique member of this orbit using the Lie cyclide splitting
\[ \underline{\mathbb{R}}^{4,2} = S_{1}\oplus S_{2}.\]
This then induces a splitting 
\[ \wedge^{2}\underline{\mathbb{R}}^{4,2}=\mathfrak{h}\oplus \mathfrak{m},\]
where 
\[ \mathfrak{h} := (S_{1}\wedge S_{1}) \oplus (S_{2}\wedge S_{2}) \quad \text{and}\quad \mathfrak{m}:=S_{1}\wedge S_{2}.\]
Thus, given a closed 1-form $\eta\in \Omega^{1}(f\wedge f^{\perp})$, we may write $\eta = \eta_{\mathfrak{h}}+\eta_{\mathfrak{m}}$,
where $\eta_{\mathfrak{h}} \in \Omega^{1}(\mathfrak{h}\cap (f\wedge f^{\perp}))$ and  $\eta_{\mathfrak{m}} \in \Omega^{1}(\mathfrak{m}\cap (f\wedge f^{\perp}))$. 

\begin{proposition}
$\eta_{\mathfrak{h}}$ is well-defined on gauge orbits.
\end{proposition}
\begin{proof}
This follows from the fact that $d\tau \in \Omega^{1}(\mathfrak{m}\,\cap (f\wedge f^{\perp}))$, for any $\tau\in \Gamma (\wedge^{2}f)$. 
\end{proof}

\begin{proposition}
\label{prop:midcon}
Modulo $\Omega^{1}(\wedge^{2} f)$, $\eta_{\mathfrak{m}} = d\tau$ for some $\tau\in \Gamma(\wedge^{2} f)$. 
\end{proposition}
\begin{proof}
Let $\sigma_{1}\in\Gamma s_{1}$ and $\sigma_{2}\in\Gamma s_{2}$ be lifts of the curvature spheres. Then we may write
\[ \eta = (\alpha_{1}\, \sigma_{1}\wedge d\sigma_{1} + \alpha_{2}\, \sigma_{2}\wedge d\sigma_{2} + 
\beta_{1}\, \sigma_{1}\wedge d\sigma_{2} + \beta_{2}\,\sigma_{2}\wedge d\sigma_{1}) \, \bmod\, \Omega^{1}(\wedge^{2} f),\]
where $\alpha_{1},\alpha_{2},\beta_{1},\beta_{2}$ are smooth functions. In this case 
\[ \eta_{\mathfrak{m}} = (\beta_{1}\,\sigma_{1}\wedge d\sigma_{2} + \beta_{2}\,\sigma_{2}\wedge d\sigma_{1}) \, \bmod\, \Omega^{1}(\wedge^{2}f).\]
Now $d\sigma_{1}\curlywedge d\sigma_{1}, d\sigma_{2}\curlywedge d\sigma_{2} \in \Omega^{2}(f\wedge f^{\perp})$ and thus 
\[ 0 = d\eta  = \beta_{1}\, d\sigma_{1}\curlywedge d\sigma_{2} + \beta_{2}\, d\sigma_{2}\curlywedge d\sigma_{1}\, \bmod\,  \Omega^{2}(f\wedge f^{\perp}).\]
Therefore $\beta_{2} = -\beta_{1}$ and
\[ \eta_{\mathfrak{m}} = d(\beta_{1} \sigma_{1}\wedge\sigma_{2}) \, \bmod \, \Omega^{1}(\wedge^{2} f).\]
Hence the result is proved.
\end{proof} 

From Proposition~\ref{prop:midcon} one can deduce that there exists a unique gauge potential of $[\eta]$ with $\eta_{\mathfrak{m}}\in\Omega^{1}(\wedge^{2} f)$, thus motivating the following definition\footnote{This potential also has a characterisation in terms of Lie algebra homology, analogous to the characterisation given in~\cite[\S2.4.1]{C2012i} for projectively applicable surfaces.}:

\begin{definition}
We call the unique gauge potential in $[\eta]$ with $\eta_{\mathfrak{m}} \in\Omega^{1}(\wedge^{2} f)$ the middle potential and denote it $\eta^{mid}$.
\end{definition}

Since $q$ is well-defined on gauge orbits, we may compute it using the middle potential. Then it is clear  that $q(X,Y) = \text{tr}(\sigma\mapsto \eta_{\mathfrak{h}}(X)d_{Y}\sigma)$, since $(\wedge^{2}f)f^{\perp}=0$.

\begin{remark}
It should be noted that it is possible for a Legendre immersion to be Lie applicable in more than one way, i.e., for there to exists more than one gauge orbit of non-trivial closed 1-forms with values in $f\wedge f^{\perp}$. The case that a Legendre immersion is Lie applicable in three parameters worth of ways has been studied in~\cite{F2002, MN2006}.
\end{remark}

\subsection{Invariant approach}
\label{subsec:inv}
We will now obtain a characterisation of Lie applicability by the existence of a certain quadratic differential. So let us assume that $q$ is a quadratic differential with respect to the conformal structure $c$, i.e., $q\in \Gamma( (T_{1}^{*})^{2}\oplus (T_{2}^{*})^{2})$. For the rest of this section we make the assumption that the signature of $q$ is constant\footnote{In order to establish a global theory of Lie applicable surfaces we will have to weaken this assumption.} over $\Sigma$. Thus, up to rescaling $q$ by $\pm 1$ and reordering $T_{1}$ and $T_{2}$, we may assume that 
\[ q= - \epsilon^{2} q_{1} + q_{2},\]
where $\epsilon\in \{0,1,i\}$, and $q_{1}\in \Gamma (T_{1}^{*})^{2}$ and $q_{2}\in \Gamma (T_{2}^{*})^{2}$ are positive definite quadratic forms. Then $q_{1}$ and $q_{2}$ determine unique lifts $\sigma_{1}\in\Gamma s_{1}$ and $\sigma_{2}\in\Gamma s_{2}$ (up to sign) such that
\[
q_{1} = (d\sigma_{2},d\sigma_{2}) \quad \text{and} \quad q_{2} = (d\sigma_{1},d\sigma_{1}). 
\]
Thus, $q$ determines a unique 1-form $\eta_{\mathfrak{h}}\in \Omega^{1}(\mathfrak{h}\cap (f\wedge f^{\perp}))$ such that 
\[q(X,Y) = \text{tr}(\sigma\mapsto \eta_{\mathfrak{h}}(X)d_{Y}\sigma),\]
namely,
\[ \eta_{\mathfrak{h}} =- \sigma_{1}\wedge d_{2}\sigma_{1} + \epsilon^{2} \sigma_{2}\wedge d_{1}\sigma_{2},\]
where we recall that each $d_{i}$ denotes the partial connection of $d$ along $T_{i}$. Let $\omega := \omega_{1} + \omega_{2}$ be a 1-form, with $\omega_{1}\in\Gamma T_{1}^{*}$ and $\omega_{2}\in\Gamma T_{2}^{*}$, and define 
\[ \eta^{mid} := - \sigma_{1}\wedge d\sigma_{1} + \epsilon^{2} \sigma_{2}\wedge d\sigma_{2} + \omega \sigma_{1}\wedge \sigma_{2}.\]
Then $\eta^{mid}$ is closed if and only if
\begin{eqnarray}
\label{eqn:midclosed}
0 =  - d\sigma_{1}\curlywedge d\sigma_{1} + \epsilon^{2} d\sigma_{2}\curlywedge d\sigma_{2} + d\omega\, \sigma_{1}\wedge \sigma_{2} - \omega \wedge d(\sigma_{1}\wedge \sigma_{2}).
\end{eqnarray}
Now let $\alpha,\beta\in \Gamma T_{1}^{*}$ and $\gamma ,\delta\in \Gamma T_{2}^{*}$ such that 
\[ d_{1}\sigma_{1} = \alpha \sigma_{1} + \beta \sigma_{2} \quad \text{and}\quad 
d_{2}\sigma_{2} = \gamma \sigma_{1} + \delta \sigma_{2}.\]

\begin{remark}
\label{rem:degcase}
Obviously, in the case that $\epsilon =0$, $q_{1}$ and thus our lift $\sigma_{2}$ of $s_{2}$ may be chosen arbitrarily. To simplify the following analysis, we will fix $q_{1}$ by choosing a lift $\sigma_{2}$ so that $\delta=0$. Note that this choice is unique up to multiplication by a smooth function $g$ such that $d_{2}g=0$. 
\end{remark}

Therefore, (\ref{eqn:midclosed}) is equivalent to 
\begin{align*}
0= &- 2\alpha\wedge (\sigma_{1}\wedge d_{2}\sigma_{1}) + 2\epsilon^{2}\delta\wedge (\sigma_{2}\wedge d_{1}\sigma_{2}) + (2\epsilon^{2}\gamma - \omega_{2})\wedge (\sigma_{1}\wedge d_{1}\sigma_{2}) \\
& + (-2\beta + \omega_{1})\wedge (\sigma_{2}\wedge d_{2}\sigma_{1}) + 
(d\omega - \omega_{1}\wedge \delta - \omega_{2}\wedge \alpha) \sigma_{1}\wedge \sigma_{2}.
\end{align*} 
Hence, $\eta^{mid}$ is closed if and only if the following two conditions hold:
\begin{enumerate}[(a)]
\item $\alpha=\delta=0$, that is, $d_{1}\sigma_{1}\in \Gamma T_{1}^{*}\otimes s_{2}$ and $d_{2}\sigma_{2}\in \Gamma T_{2}^{*}\otimes s_{1}$.
\item\label{item:invdem} $\omega = 2(\beta + \epsilon^{2} \gamma)$ and $\omega$ is closed.
\end{enumerate}
These two conditions can be reformulated as conditions on $q$. In Lemma~\ref{app:quad} we show that the first condition is equivalent to $q$ being divergence-free with respect to the conformal structure $c$ on $T\Sigma$. In other words, in terms of conformal curvature line coordinates $(u,v)$, there exist functions $U$ of $u$ and $V$ of $v$ such that 
\[ q = -\epsilon^{2} U^{2}du^{2} + V^{2}dv^{2}.\]
The second condition can be equated to a condition on the Darboux cubic form.
Recall that we defined the Darboux cubic form $\mathcal{C}\in \Gamma S^{3}T^{*}\Sigma \otimes (\wedge^{2}f)^{*}$ as
\[ \mathcal{C}(X,Y,Z)\sigma\wedge\nu = (\mathcal{D}_{X}\mathcal{D}_{Y}\sigma,\mathcal{N}_{Z}\nu) - 
(\mathcal{D}_{X}\mathcal{D}_{Y}\nu,\mathcal{N}_{Z}\sigma),\]
where $\sigma,\nu\in\Gamma f$ and $X,Y,Z\in\Gamma T\Sigma$. Then in terms of the special lifts $\sigma_{1}$ and $\sigma_{2}$,
\begin{align} 
\mathcal{C}(X,Y,Z)\sigma_{1}\wedge \sigma_{2} &= -\gamma(Z)(d_{Y}\sigma_{1},d_{X}\sigma_{1})
+ \beta(Z)(d_{Y}\sigma_{2},d_{X}\sigma_{2})\label{eqn:dcubq}\\
&= -\gamma(Z)q_{2}(X,Y) + \beta(Z) q_{1}(X,Y). \nonumber
\end{align}
Now let $X\in\Gamma T_{1}$, $Y\in\Gamma T_{2}$ such that 
\[ q_{1}(X,X) = q_{2}(Y,Y) = 1.\]
Then we may define a 1-form
\[ \mathcal{C}^{q} := (\mathcal{C}(X,X,.) -\epsilon^{2}\mathcal{C}(Y,Y,.))\sigma_{1}\wedge \sigma_{2}.\]
It is then clear from Equation~(\ref{eqn:dcubq}) that $\mathcal{C}^{q}  = \beta + \epsilon^{2}\gamma$ and therefore condition~(\ref{item:invdem}) is equivalent to the closure\footnote{
In the case that $\epsilon =0$, $X$ and thus $\mathcal{C}^{q}$ are determined by our choice of lift of $s_{2}$ in Remark~\ref{rem:degcase}. A different choice of such a lift scales $\mathcal{C}^{q}$ by a function $g$ satisfying $d_{2}g =0$. Therefore, the closure of $\mathcal{C}^{q}$ is not affected by this choice.} of $\mathcal{C}^{q}$. We have thus arrived at the following theorem:

\begin{theorem}
\label{thm:invariant}
An umbilic-free Legendre map $f$ is an $\Omega$-surface ($\Omega_{0}$-surface) if and only if there exists a non-zero divergence-free, non-degenerate (degenerate) quadratic differential (with respect to the conformal structure $c$ induced by $f$) $q$ such that $\mathcal{C}^{q}$ is closed.
\end{theorem}

\begin{remark}
In~\cite[\S2.4.1]{C2012i}, by using Lie algebra homology, an elegant characterisation of projectively applicable surfaces is given in terms of a quadratic differential and the Darboux cubic form. An analogous homological characterisation can be obtained for Lie applicable surfaces, however, this is beyond the scope of this paper. 
\end{remark}

Condition~(\ref{item:invdem}) also tells us that the middle potential is given by 
\begin{equation} \eta^{mid} = \sigma_{1}\wedge \star d\sigma_{1} + \epsilon^{2}\sigma_{2}\wedge \star d\sigma_{2}, \label{eqn:midform}
\end{equation}
where $\star$ is the hodge star operator induced by the conformal structure $c$.

\subsection{Demoulin's equation}
Now suppose that $f$ is the lift of an umbilic-free space-form projection $\mathfrak{f}$ with tangent plane congruence $\mathfrak{t}$, i.e., $f=\langle \mathfrak{f},\mathfrak{t}\rangle$. Then, from Subsection~\ref{subsec:symbreak},
\[ \mathfrak{t}+\kappa_{1}\mathfrak{f} \quad \text{and}\quad \mathfrak{t}+\kappa_{2}\mathfrak{f}\]
are lifts of the curvature sphere congruences $s_{1}$ and $s_{2}$, respectively. Thus, there exists functions $\lambda$ and $\mu$ such that our special lifts $\sigma_{1}$ and $\sigma_{2}$ are given by 
\[ \sigma_{1} = \lambda (\mathfrak{t}+\kappa_{1}\mathfrak{f}) \quad \text{and}\quad \sigma_{2} = \mu (\mathfrak{t}+\kappa_{2}\mathfrak{f}).\]
Since $q$ is divergence-free, in terms of arbitrary curvature line coordinates $(u,v)$, there exist  functions $U$ of $u$ and $V$ of $v$ such that 
\[ q = -\epsilon^{2}U^{2}du^{2} + V^{2}dv^{2}.\]
Thus, 
\[ V^{2} = (\sigma_{1,v},\sigma_{1,v}) = \lambda^{2} ( \mathfrak{t}_{v} + \kappa_{1}\mathfrak{f}_{v},    \mathfrak{t}_{v} + \kappa_{1}\mathfrak{f}_{v}) = \lambda^{2} (\kappa_{1} - \kappa_{2})^{2} G. \]
Hence, 
\[ \lambda = \pm \frac{V}{\sqrt{G}( \kappa_{1}-\kappa_{2})}.\]
Similarly, 
\[ \mu = \pm \frac{U}{\sqrt{E}(\kappa_{1}-\kappa_{2})}.\]
On the other hand, we have that 
\[ d_{1}\sigma_{1} = \beta \sigma_{2} \quad \text{and}\quad d_{2}\sigma_{2} = \gamma \sigma_{1}.\]
Therefore, 
\[ \beta \mu (\mathfrak{t}+\kappa_{2}\mathfrak{f}) = d_{1}\lambda (\mathfrak{t}+\kappa_{1}\mathfrak{f}) + \lambda d_{1}\kappa_{1} \mathfrak{f}\]
and 
\[ \beta = - \frac{\lambda}{\mu} \frac{d_{1}\kappa_{1}}{\kappa_{1}-\kappa_{2}}.\]
Similarly, 
\[ \gamma = \frac{\mu}{\lambda} \frac{d_{2}\kappa_{2}}{\kappa_{1}-\kappa_{2}}.\]
Thus, 
\[ C^{q} = \beta+ \epsilon^{2}\gamma = \pm \left(- \frac{V\sqrt{E}}{U\sqrt{G}} \frac{d_{1}\kappa_{1}}{\kappa_{1}-\kappa_{2}} +\epsilon^{2}\frac{U\sqrt{G}}{V\sqrt{E}} \frac{d_{2}\kappa_{2}}{\kappa_{1}-\kappa_{2}} \right).\]
Hence, $C^{q}$ is closed if and only if 
\[ 0 = \left(\frac{V\sqrt{E}}{U\sqrt{G}} \frac{\kappa_{1,u}}{\kappa_{1}-\kappa_{2}}\right)_{v} +\epsilon^{2}\left(\frac{U\sqrt{G}}{V\sqrt{E}} \frac{\kappa_{2,v}}{\kappa_{1}-\kappa_{2}} \right)_{u}.\]
Thus, $f$ is an $\Omega$-/$\Omega_{0}$-surface if and only if the space form projection $\mathfrak{f}$ is an $\Omega$-/$\Omega_{0}$-surface in the sense of Demoulin~\cite{D1911ii}. 

\subsection{Isothermic sphere congruences} 

We will now see how Lie applicable surfaces envelop isothermic sphere congruences. We say that a sphere congruence is isothermic if it is isothermic as a surface in the Lie quadric (with respect to the natural conformal structure on the Lie quadric). Equivalently, we have the following definition:

\begin{definition}[\cite{BDPP2011,H2003}]
A sphere congruence $s:\Sigma\to \mathbb{P}(\mathcal{L})$ is isothermic if there exists a non-zero closed 1-form $\eta_{s}\in \Omega^{1}(s\wedge s^{\perp})$. 
\end{definition}

Now suppose that $f$ is an umbilic-free Lie applicable surface with middle potential 
\[ \eta^{mid} = \sigma_{1}\wedge \star d\sigma_{1} + \epsilon^{2}\sigma_{2}\wedge \star d\sigma_{2}.\]
Then we may gauge $\eta^{mid}$ by $\pm \epsilon \sigma_{1}\wedge  \sigma_{2}$ to obtain\footnote{
Notice that $\eta^{mid} = \frac{1}{2}(\eta^{+}+ \eta^{-})$. This is our justification for calling $\eta^{mid}$ the middle potential.}
\[ \eta^{\pm} := \eta^{mid} + d(\pm\epsilon \sigma_{1}\wedge \sigma_{2}) = (\sigma_{1}\pm \epsilon\sigma_{2})\wedge \star d(\sigma_{1}\pm \epsilon\sigma_{2})  \in \Omega^{1}(s^{\pm}\wedge (s^{\pm})^{\perp}),\]
where $s^{\pm} := \langle \sigma_{1}\pm \epsilon\sigma_{2}\rangle$. Hence, $s^{\pm}$ are isothermic sphere congruences. In the case that $\epsilon \neq 0$ we have that $s^{\pm}$ are a pair of isothermic sphere congruences separating the curvature sphere congruences $s_{1}$ and $s_{2}$ harmonically. If $\epsilon=0$ we have that the curvature sphere congruence $s_{1}$ is isothermic. 

\begin{theorem}
If $f$ is an umbilic-free $\Omega$-surface then $f$ envelops a pair of isothermic sphere congruences that separate the curvature sphere congruences harmonically. Furthermore, if $q$ is indefinite then the isothermic sphere congruences are real and if $q$ is positive definite then they are complex conjugate. 

If $f$ is an umbilic-free $\Omega_{0}$ surface then $f$ envelops a curvature sphere congruence that is isothermic. 
\end{theorem}

\begin{lemma}
\label{lem:pointiso}
Let $s\le f$ be a sphere congruence enveloped by $f$ and suppose that there exists $\eta \in [\eta^{mid}]$ such that at a point $p\in\Sigma$ 
\[ \eta_{p}\in T_{p}^{*}\Sigma\otimes (s(p)\wedge f(p)^{\perp}).\]
Then $s$ coincides with one of the isothermic sphere congruences enveloped by $f$ at $p$. 
\end{lemma}
\begin{proof}
Since $\eta\in[\eta^{mid}]$, there exists a smooth function $\lambda$ such that 
\[ \eta = \eta^{mid} + d(\lambda \sigma_{1}\wedge \sigma_{2}).\]
Now using that 
\[ \eta^{mid} = \sigma_{1}\wedge \star d\sigma_{1} + \epsilon^{2} \sigma_{2}\wedge \star d\sigma_{2}\]
we have that 
\[ \eta = \sigma_{1}\wedge(\lambda d_{1}\sigma_{2} - d_{2}\sigma_{1}) + \sigma_{2}\wedge (\epsilon^{2} d_{1}\sigma_{2} - \lambda d_{2}\sigma_{1}) \, \bmod\, \Omega^{1}(\wedge^{2}f).\]
Since $d_{1}\sigma_{2}$ and $d_{2}\sigma_{1}$ are linearly independent, $\eta$ nowhere takes values in $s_{2}\wedge f^{\perp}$, for all smooth functions $\lambda$. Therefore, let $\mu\in \mathbb{R}$ such that $\sigma(p)=\sigma_{1}(p)+\mu \sigma_{2}(p)$ is a lift of $s(p)$. Then 
\[ \eta_{p}\in T_{p}^{*}\Sigma\otimes (s(p)\wedge f(p)^{\perp})\]
if and only if 
\[ \mu (\lambda(p) d_{1}\sigma_{2} - d_{2}\sigma_{1}) = \epsilon^{2}d_{1}\sigma_{2} - \lambda(p)d_{2}\sigma_{1}.\]
Since $d_{1}\sigma_{2}$ and $d_{2}\sigma_{1}$ are linearly independent at $p$, this is equivalent to 
\[ \mu = \lambda(p) \quad \text{and}\quad \lambda(p)^{2} = \epsilon^{2}.\]
Thus, $\sigma(p)=\sigma_{1}(p) \pm\epsilon \sigma_{2}(p)\in s^{\pm}(p)$. 
\end{proof}

\subsubsection{The $\Delta_{q}$ operator}
\label{subsubsec:Dq}
Let $X\in\Gamma T_{1}$ and $Y\in\Gamma T_{2}$ such that 
\[ q_{1}(X,X)=1\quad \text{and}\quad q_{2}(Y,Y)=1.\]
Then we define an operator 
\[ \Delta_{q} := d_{X}d_{X} -\epsilon^{2} d_{Y}d_{Y}.\]
Using $\Delta_{q}$ we define a map $\zeta_{q}:f\otimes f\to \mathbb{R}$ by
\[ \zeta_{q}(\nu,\xi) = (\Delta_{q}\nu,\xi).\]
Then $\zeta_{q}$ is a symmetric tensor and identifies the isothermic sphere congruences:

\begin{proposition}
\label{prop:isozeta}
Let $s\le f$. Then $\zeta_{q}(s(p),s(p)) =0$ if and only if $s$ coincides with one of the isothermic sphere congruences at $p$. 
\end{proposition}
\begin{proof}
Let $\sigma_{1}$ and $\sigma_{2}$ be the special lifts of the curvature spheres $s_{1}$ and $s_{2}$, respectively, such that 
\[ q_{1} = (d\sigma_{2},d\sigma_{2}) \quad \text{and}\quad q_{2}=(d\sigma_{1},d\sigma_{1}).\]
Since $s_{1}$ and $s_{2}$ are curvature spheres, we have that 
\[ \Delta_{q}\sigma_{1} = -\epsilon^{2}d_{Y}d_{Y}\sigma_{1}\,\bmod\, f^{\perp}\quad \text{and}\quad \Delta_{q}\sigma_{2} = d_{X}d_{X}\sigma_{2}\,\bmod\, f^{\perp}.\]
Let $\sigma\in\Gamma s$ and let $\alpha$ and $\beta$ be smooth functions such that $\sigma=\alpha\sigma_{1}+\beta\sigma_{2}$. Then 
\begin{eqnarray*} 
\zeta_{q}(\sigma,\sigma) = \beta^{2}(d_{X}d_{X}\sigma_{2},\sigma_{2}) -\epsilon^{2}\alpha^{2} (d_{Y}d_{Y}\sigma_{1},\sigma_{1})
=-\beta^{2} +\epsilon^{2}\alpha^{2}.
\end{eqnarray*}
Thus, $\zeta_{q}(\sigma,\sigma)=0$ if and only if $\beta = \pm\epsilon \alpha$, which holds if and only if $\sigma\in \Gamma s^{\pm}$. Since $\zeta_{q}$ is tensorial, this is a pointwise condition. 
\end{proof}

\subsubsection{Christoffel dual lifts}
Suppose that $\epsilon\neq 0$. Recall that $\mathcal{C}^{q}$ is a closed 1-form. Thus, there exist non-trivial functions $\xi^{\pm}$ such that 
\[ d\xi^{\pm} = \mp \epsilon^{-1}\, \mathcal{C}^{q} \xi^{\pm}.\]
Now, $\xi^{+}\xi^{-}$ is constant, and, without loss of generality we will assume that $\xi^{+}\xi^{-}=-1$. We may then define unique (up to reciprocal constant rescaling) lifts $\sigma^{\pm}$ of the isothermic sphere congruences $s^{\pm}$ by 
\[ \sigma^{\pm} := \xi^{\pm}(\sigma_{1}\pm \epsilon \sigma_{2}).\]
A straightforward computation shows that:
\begin{proposition}
\label{prop:chrlifts}
$\eta^{\pm} = \sigma^{\pm}\wedge d\sigma^{\mp}$ and $d\sigma^{+}\curlywedge d\sigma^{-}=0$. 
\end{proposition} 
We call these lifts the \textit{Christoffel dual lifts} of $s^{\pm}$. 

\section{Transformations of Lie applicable surfaces}
\label{sec:trafos}
In this section we shall review and expand on the transformation theory of Lie applicable surfaces presented by Clarke \cite{C2012i}. In particular we shall show how the middle potential behaves under such transformations. 

Suppose that $f$ is a Lie applicable surface with closed 1-form $\eta$. 

\begin{theorem}[{\cite[Lemma 4.2.6]{C2012i}}]
$\{d+t\eta\}_{t\in \mathbb{R}}$ is a 1-parameter family of flat metric connections. 
\end{theorem} 
\begin{proof}
The curvature of the connection $d+t\eta$ is given by 
\[ R^{d+t\eta} = t d\eta + \frac{t^{2}}{2}[\eta\wedge \eta] =0.\]
The fact that $d+t\eta$ is a metric connection follows from the skew-symmetry of $\eta$.  
\end{proof}

Our choice of $\eta$ in the gauge orbit was arbitrary, so it is prudent to examine how these connections change when we use a different member of the gauge orbit. Suppose that $\tilde{\eta} = \eta - d\tau$ for some $\tau \in \Gamma (\wedge^{2}f)$. Then a straightforward computation shows that:

\begin{lemma}[{\cite[Lemma 4.5.1]{C2012i}}] 
\label{lem:flatgorb} $d+t\tilde{\eta} = \exp(t\tau)\cdot (d+t\eta).$
\end{lemma}

\subsection{Calapso transforms}
\label{subsec:cal}
Since $\{d^{t}:=d+t\eta\}_{t\in\mathbb{R}}$ is a 1-parameter family of flat metric connections, for each $t\in\mathbb{R}$, there exists a local orthogonal trivialising gauge transformation $T(t):\Sigma\to O(4,2)$, i.e., 
\begin{equation}
\label{eqn:calgauge}
T(t)\cdot d^{t} =d.
\end{equation}
\begin{definition}
$f^{t}:= T(t)f$ is called a Calapso transform of $f$. 
\end{definition}
Now suppose that $\tilde{\eta} = \eta -d\tau$, and let $\widetilde{T}(t)$ denote the corresponding local orthogonal trivialising gauge transformations. Then from Lemma~\ref{lem:flatgorb}, it follows that 
\[\widetilde{T}(t) = T(t)\exp(-t\tau).\]
Since $(\wedge^{2}f)f=0$, it follows that the Calapso transforms are well defined on the gauge orbit $[\eta]$. 

Let $\sigma^{t} := T(t)\sigma$ be a section of $f^{t}$. Then by Equation~(\ref{eqn:calgauge}),
\begin{equation*} 
d\sigma^{t} = d(T(t)\sigma) = T(t)(d+t\eta)\sigma = T(t)d\sigma. 
\end{equation*}
From this one can easily deduce that the contact and immersion conditions hold for $f^{t}$ and thus $f^{t}$ is a Legendre immersion. Moreover, we can deduce that if $s(p)$ is a curvature sphere of $f$ at $p$ then $s^{t}(p):= T(t)s(p)$ is a curvature sphere of $f^{t}$ at $p$ and the corresponding curvature spaces coincide.

\begin{theorem}
\label{thm:calform}
$\eta^{t}:= Ad_{T(t)}\cdot\eta$
is a closed 1-form with values in $\Omega^{1}(f^{t}\wedge (f^{t})^{\perp})$ with $[\eta^{t}\wedge \eta^{t}]=0$ and $q^{t}=q$.
Hence, $f^{t}$ is a Lie applicable surface. 
\end{theorem}
\begin{proof}
The closedness of $\eta^{t}$ follows from 
\[ d\eta^{t} = (T(t)\cdot d^{t}) Ad_{T(t)}\cdot\eta = T(t)\cdot d^{t}\eta=T(t)\cdot(d\eta + t[\eta\wedge \eta])=0.\]
Furthermore, 
\[ [\eta^{t}\wedge \eta^{t}] = Ad_{T(t)}\cdot [\eta\wedge \eta]=0.\]
Finally, for $\sigma^{t}:= T(t)\sigma$
\[ \eta^{t}(X)d_{Y}\sigma^{t} = (Ad_{T(t)}\cdot\eta(X)) (T(t)\cdot (d+t\eta)(Y))\sigma^{t} = T(t)\eta(X)d_{Y}\sigma.\]
Thus, 
\[ q^{t}(X,Y) = tr(\sigma^{t}\mapsto \eta^{t}(X)d_{Y}\sigma^{t})\]
coincides with $q(X,Y)$ for all $X,Y\in\Gamma T\Sigma$. 
\end{proof}

We will now see how the 1-parameter family of flat connections of a Calapso transform are related to those of the original surface:

\begin{proposition}
\label{prop:calgauge}
For any $s\in\mathbb{R}$, 
\[ d+s\eta^{t} = T(t)\cdot (d+(s+t)\eta).\]
Therefore the local trivialising orthogonal gauge transformations of $d+s\eta^{t}$ are 
\[ T^{t}(s) = T(s+t)T^{-1}(t).\]
\end{proposition}
\begin{proof}
Using Theorem~\ref{thm:calform}, we have that 
\[ d+s\eta^{t} = d+sAd_{T(t)}\cdot\eta = T(t)\cdot( T^{-1}(t) \cdot d + s\eta) = T(t)\cdot( d+(s+t)\eta),\]
and the result follows. 
\end{proof}

From Proposition~\ref{prop:calgauge} we can quickly deduce the analogue of the permutability result of Hertrich-Jeromin~\cite[\S5.5.9]{H2003} for Calapso transforms of isothermic surfaces:
\[T^{t}(s) T(t) = T(s+t).\]

Now let us assume that $f$ is umbilic-free and we are using the middle potential, i.e., $\eta = \eta^{mid}$. 

\begin{lemma}
\label{lem:calcyc}
The Lie cyclides of $f^{t}$ are given by 
\[ S_{1}^{t} = T(t)S_{1} \quad \text{and} \quad S_{2}^{t} = T(t)S_{2}.\]
Hence, the induced splitting of the trivial connection $d = \mathcal{D}^{t}+\mathcal{N}^{t}$ satisfies
\[ \mathcal{D}^{t} = T(t)\cdot (\mathcal{D} +t\eta_{\mathfrak{h}}) \quad \text{and}\quad \mathcal{N}^{t} = T(t)\cdot (\mathcal{N} +t\eta_{\mathfrak{m}}).\]
\end{lemma}
\begin{proof}
Let $Y\in\Gamma T_{2}$ and $\sigma_{1}^{t}=T(t)\sigma_{1}$ be a lift of the curvature sphere $s_{1}^{t}$. Then 
\[ d_{Y}\sigma_{1}^{t} = d_{Y}(T(t)\sigma_{1}) = (T(t)\cdot d^{t}_{Y})T(t)\sigma_{1} = T(t)(d^{t}_{Y}\sigma_{1}) = T(t)d_{Y}\sigma_{1},\]
since $\eta^{mid} f=0$. Thus, $d\sigma^{t}_{1}(T_{2}) = T(t)d\sigma_{1}(T_{2})$. Furthermore, 
\[ d_{Y}d_{Y}\sigma_{1}^{t}= d_{Y}d_{Y}(T(t)\sigma_{1}) = (T(t)\cdot d_{Y}^{t})(T(t)d_{Y}\sigma_{1}) = T(t)d_{Y}^{t}d_{Y}\sigma_{1}.\]
Now, since we are using the middle potential, $\eta^{mid}(Y)d_{Y}\sigma_{1}\in \Gamma s_{1}$. Thus, 
$d_{Y}d_{Y}\sigma_{1}^{t} \in \Gamma T(t)S_{1}$ and
\[ S_{1}^{t} = s^{t}_{1}\oplus d\sigma^{t}_{1}(T_{2}) \oplus \langle d_{Y}d_{Y}\sigma_{1}^{t} \rangle = T(t)S_{1}.\]
Similarly, $S_{2}^{t}= T(t)S_{2}$. From
\[ d = T(t)\cdot (d+t\eta^{mid}) = T(t)\cdot (\mathcal{D} + \mathcal{N} + t\eta_{\mathfrak{h}}+ t\eta_{\mathfrak{m}}), \]
one can deduce the remainder of the lemma. 
\end{proof}

Blaschke~\cite{B1929} showed that Lie applicable surfaces are the only surfaces that are not determined by the Lie invariant metric and Darboux cubic form. Therefore the following corollary comes as no surprise: 

\begin{corollary}
\label{cor:calmetric}
The Lie-invariant metric $g^{L}$ is preserved by Calapso transform and
the Darboux cubic form $\mathcal{C}^{t}\in \Gamma (S^{3}T^{*}\Sigma\otimes (\wedge^{2}f^{t})^{*})$ of $f^{t}$ satisfies
\begin{eqnarray} 
\label{eqn:calcub}
\mathcal{C}^{t}\circ T(t) = \mathcal{C},
\end{eqnarray}
that is, for $\tau\in\Gamma (\wedge^{2}f)$ and $X,Y,Z\in\Gamma T\Sigma$, 
\[ \mathcal{C}^{t}(X,Y,Z)(T(t)\cdot \tau) = \mathcal{C}(X,Y,Z)\tau.\]
\end{corollary}

\begin{corollary}
The middle potential of $f^{t}$ is $(\eta^{t})^{mid} = Ad_{T(t)}\cdot \eta^{mid}$. 
\end{corollary}
\begin{proof}
In Section~\ref{sec:lieapp} we had a splitting $\underline{\mathfrak{o}(4,2)}=\mathfrak{h}+\mathfrak{m}$ induced by $f$, where
\[ \mathfrak{h} = (S_{1}\wedge S_{1})\oplus (S_{2}\wedge S_{2}) \quad \text{and}\quad \mathfrak{m}=S_{1}\wedge S_{2}. \]
By Lemma~\ref{lem:calcyc}, $f^{t}$ induces the splitting $\underline{\mathfrak{o}(4,2)}=\mathfrak{h}^{t}+\mathfrak{m}^{t}$, where 
\[ \mathfrak{h}^{t} = T(t)\cdot \mathfrak{h} \quad \text{and} \quad \mathfrak{m}^{t}=T(t)\cdot \mathfrak{m}.\]
We then split $\eta^{mid} = \eta_{\mathfrak{h}}+\eta_{\mathfrak{m}}$, where $\eta_{\mathfrak{h}}\in\Omega^{1}(\mathfrak{h})$ and $\eta_{\mathfrak{m}}\in\Omega^{1}(\mathfrak{m})$. Now splitting $Ad_{T(t)}\cdot \eta^{mid}$ with respect to the splitting induced by $f^{t}$ yields $Ad_{T(t)}\cdot \eta^{mid} = \eta^{t}_{\mathfrak{h}^{t}}+ \eta^{t}_{\mathfrak{m}^{t}}$ with 
\[ \eta^{t}_{\mathfrak{h}^{t}} = Ad_{T(t)}\cdot\eta_{\mathfrak{h}} \quad \text{and}\quad \eta^{t}_{\mathfrak{m}^{t}}  = Ad_{T(t)}\cdot\eta_{\mathfrak{m}}.\]
Since $\eta^{mid}$ is the middle potential, $\eta_{\mathfrak{m}} \in\Omega^{1}(f\wedge f)$. Hence, 
\[ \eta^{t}_{\mathfrak{m}^{t}}  = Ad_{T(t)}\cdot\eta_{\mathfrak{m}} \in\Omega^{1}(f^{t}\wedge f^{t}).\]
Therefore $Ad_{T(t)}\cdot \eta^{mid}$ is the middle potential of $f^{t}$. 
\end{proof}

\begin{proposition}
Suppose that $s$ is an isothermic sphere congruence of $f$ with isothermic gauge potential $\eta^{s}\in\Omega^{1}(s\wedge s^{\perp})$. Then $s^{t}:=T(t)s$ is an isothermic\footnote{$s^{t}$ is in fact the Calapso transform of the isothermic sphere congruence $s$, see~\cite{BDPP2011,H2003}.} sphere congruence of $f^{t}$ with isothermic gauge potential $(\eta^{t})^{s} := Ad_{T(t)}\cdot \eta^{s}$. 
\end{proposition}
\begin{proof}
From the orthogonality of $T(t)$ we have
\[ T(t)\cdot s\wedge s^{\perp} = s^{t}\wedge (s^{t})^{\perp}.\]
Hence, $(\eta^{t})^{s}\in \Omega^{1}(s^{t}\wedge (s^{t})^{\perp})$ and $s^{t}$ is isothermic.
\end{proof}

\subsection{Darboux transforms}
\label{subsec:darb}
Fix a non-zero $m\in\mathbb{R}$. Since $d^{m}=d+m\eta$ is a flat connection, it has many parallel sections. Suppose that $\hat{s}$ is a null rank 1 parallel subbundle of $d^{m}$ such that $\hat{s}$ is nowhere orthogonal to the curvature sphere congruences of $f$. Let $s_{0}: =  \hat{s}^{\perp}\cap f$ and let $\hat{f}:= s_{0}\oplus \hat{s}$. 

\begin{definition}
$\hat{f}$ is a Darboux transform of $f$ with parameter $m$. 
\end{definition}

Now for any section $\sigma_{0}\in\Gamma s_{0}$ and any parallel section $\hat{\sigma}\in\Gamma \hat{s}$ of $d^{m}$
\begin{equation} 
d\sigma_{0},\, d\hat{\sigma}\in \Omega^{1}((f+\hat{f})^{\perp}).\label{eqn:orthsec}
\end{equation}
It is then clear that $\hat{f}$ satisfies the contact condition. It remains to check the immersion condition of $\hat{f}$: let $p\in\Sigma$ and suppose that there exists $X\in T_{p}\Sigma$ such that $d_{X}\sigma_{0}\in \hat{f}(p)$ for some lift $\sigma_{0}\in\Gamma s_{0}$. Then as $d\sigma_{0}\in \Omega^{1}((f+\hat{f})^{\perp})$, we have that  $d_{X}\sigma_{0} \in s_{0}(p)$. Then it follows from the fact that $s_{0}$ is nowhere a curvature sphere of $f$ that $X=0$.  Therefore, $\hat{f}$ is a Legendre immersion. 

Recall from Section~\ref{sec:ribaucour} that we defined Ribaucour transforms of Legendre immersions. 
\begin{lemma}
$\hat{f}$ is a Ribaucour transform of $f$.
\end{lemma}
\begin{proof}
By Equation~(\ref{eqn:orthsec}), for a parallel section $\hat{\sigma}\in\Gamma \hat{s}$ of $d^{m}$,
\[ d\hat{\sigma}\in \Omega^{1}((f+\hat{f})^{\perp}).\]
Therefore, $\hat{\sigma}\, \bmod\, s_{0}$ is a parallel section of the induced connection on $(f+\hat{f})/s_{0}$. Hence, this connection is flat.
\end{proof}
Suppose that $s\le f$ is a rank 1 subbundle of $f$ such that $s\cap s_{0}=\{0\}$ and define $l:=s\oplus\hat{s}$. Then $l$ defines a $(1,1)$-subbundle of $\underline{\mathbb{R}}^{4,2}$ and we have  the following splitting of $\underline{\mathbb{R}}^{4,2}$:
\[ \underline{\mathbb{R}}^{4,2} = l\oplus l^{\perp}.\]
We can then use this splitting to split the trivial connection $d$ on $\underline{\mathbb{R}}^{4,2}$ into 
\[ d = \mathcal{D}^{l,l^{\perp}}+\mathcal{N}^{l,l^{\perp}},\]
where $\mathcal{D}^{l,l^{\perp}}$ is the sum of the induced connections $\mathcal{D}^{l}$ and $\mathcal{D}^{l^{\perp}}$ on $l$ and $l^{\perp}$, respectively, and $\mathcal{N}^{l,l^{\perp}} \in \Omega^{1}(l\wedge l^{\perp})$. By Corollary~\ref{cor:lflat}, $\mathcal{D}^{l}$ is a flat connection on $l$ and if $\hat{\sigma}$ is a parallel section of $d^{m}$, then $\hat{\sigma}$ is a parallel section of $\mathcal{D}^{l}$. We may further split $\mathcal{N}^{l,l^{\perp}} =- \beta -\hat{\beta}$ where 
\[ \beta\in\Omega^{1}(\hat{s}\wedge l^{\perp})\quad \text{and} \quad \hat{\beta}\in\Omega^{1}(s\wedge l^{\perp}). \]
Moreover we may use our splitting to split $\eta = \eta_{0} + \eta_{s}$, where 
\[ \eta_{0}\in \Omega^{1}(s_{0}\wedge l^{\perp}) \quad \text{and} \quad \eta_{s}\in \Omega^{1}(s\wedge l^{\perp}).\] 

Recall from~\cite{BDPP2011,BS2012,C2012i,S2008} that for $v,w\in \mathcal{L}$ such that $(v,w)\neq 0$ and non-zero $t\in\mathbb{R}$ we have an orthogonal transformation  
\[ \Gamma^{v}_{w}(t)u = \left\{ \begin{array}{ll}
t\, u &\text{for $u=v$,}\\
\frac{1}{t}\, u &\text{for $u=w$,}\\
u &\text{for $u\in\langle v,w\rangle^{\perp}$.}
\end{array}\right.\]
We are now in a position to state the following proposition:
\begin{proposition}
\label{prop:hateta}
There exists a closed 1-form $\hat{\eta}\in\Omega^{1}(\hat{f}\wedge \hat{f}^{\perp})$ with $[\hat{\eta}\wedge \hat{\eta}]=0$ such that 
\[ d+t\hat{\eta} =\Gamma^{\hat{s}}_{s}(1-t/m) \cdot (d+t\eta).\]
Furthermore, $s$ is a parallel subbundle of $d+m\hat{\eta}$ and the quadratic differential $\hat{q}$ of $\hat{\eta}$ coincides with $q$.
\end{proposition}
\begin{proof}
The first part of this theorem was proved by Clarke \cite[Theorem 4.3.7]{C2012i} and is analogous to \cite[Proposition 3.11]{BDPP2011}. For the purpose of proving the latter part of this theorem, we shall repeat the arguments of those proofs here. 

Firstly, for a parallel section $\hat{\sigma}\in\Gamma \hat{s}$ of $d^{m}$, we have that $d\hat{\sigma} = -m\eta\hat{\sigma}$. Therefore $-\hat{\beta}\hat{\sigma} = -m \eta_{s}\hat{\sigma}$. This implies that $\hat{\beta} = m\eta_{s}$. 
Now we may write 
\[ d+ t\eta = \mathcal{D}^{l,l^{\perp}} - \beta - \hat{\beta} + t\eta_{0} + t\eta_{s}.\]
Therefore, 
\begin{align*}
\Gamma^{\hat{s}}_{s}(1-t/m) \cdot(d+t\eta) &= \Gamma^{\hat{s}}_{s}(1-t/m)\cdot (\mathcal{D}^{l,l^{\perp}} -\beta- (1 - t/m)\hat{\beta} + t \eta_{0} )\\
&= \mathcal{D}^{l,l^{\perp}} - (1-t/m)\beta - (1-t/m)/(1-t/m)\hat{\beta} + t\eta_{0} \\
&= \mathcal{D}^{l,l^{\perp}} - \hat{\beta} - \beta + t(\eta_{0} + (1/m) \beta).
\end{align*}
Then letting $\eta_{\hat{s}}:= (1/m) \beta$ and $\hat{\eta}:= \eta_{0}+\eta_{\hat{s}}\in \Omega^{1}(\hat{f}\wedge \hat{f}^{\perp})$, we have that 
\[ d+t\hat{\eta} =\Gamma^{\hat{s}}_{s}(1-t/m) \cdot (d+t\eta).\]
Since $d+t\eta$ is a 1-parameter family of flat connections, we must have that $d+t\hat{\eta}$ is a 1-parameter family of flat connections. The curvature of this family is given by 
\[ R^{d+t\hat{\eta}} = td\hat{\eta} + \frac{t^{2}}{2}[\hat{\eta}\wedge \hat{\eta}].\]
Thus, $\hat{\eta}$ is closed and $[\hat{\eta}\wedge \hat{\eta}]=0$. 

Suppose that $\sigma\in\Gamma s$ is a parallel section of $\mathcal{D}^{l}$. Then $d\sigma= -\beta\sigma$ and
\[ (d+m\hat{\eta})\sigma = -\beta\sigma + m(1/m)\beta \sigma =0.\]
Hence, $s$ is a parallel subbundle of $d+m\hat{\eta}$. 

We shall now show that the quadratic forms of $\hat{\eta}$ and $\eta$ coincide: let $\sigma_{0}\in \Gamma s_{0}$ and assume that $(\sigma,\hat{\sigma})=-1$. Now, $\{\sigma_{0},\sigma\}$ is a basis for $f$ and $\{\sigma_{0},\hat{\sigma}\}$ is a basis for $\hat{f}$. Since $\hat{\eta}= \eta_{0} +\eta_{\hat{s}}$ and $\eta=\eta_{0}+\eta_{s}$, we have that, for $X,Y\in\Gamma T\Sigma$,
\[ [\hat{\eta}(X)d_{Y}\sigma_{0}]_{s_{0}} = [\eta_{0}(X)d_{Y}\sigma_{0}]_{s_{0}}=[\eta(X)d_{Y}\sigma_{0}]_{s_{0}}.\]
Therefore, with respect to our bases defined above, the $s_{0}$ component of $\hat{\eta}(X)d_{Y}\sigma_{0}$ coincides with the $s_{0}$ component of $\eta(X)d_{Y}\sigma_{0}$.
Furthermore, the $\hat{s}$ component of $\hat{\eta}(X)d_{Y}\hat{\sigma}$ is given by 
\begin{align*}
-(\hat{\eta}(X)d_{Y}\hat{\sigma},\sigma) &= (1/m)(\beta(X)\hat{\beta}(Y)\hat{\sigma},\sigma)\\
&= 
(1/m)(\hat{\sigma},\hat{\beta}(Y)\beta(X)\sigma) =  -(\hat{\sigma},\eta(Y)d_{X}\sigma),  
\end{align*}
by the skew-symmetry of $\beta$ and $\hat{\beta}$. Therefore, the $\hat{s}$ component of $\hat{\eta}(X)d_{Y}\hat{\sigma}$ coincides with the $s$ component of $\eta(X)d_{Y}\sigma$. It follows then that 
\[ q(X,Y) = tr(\nu\mapsto \eta(X)d_{Y}\nu) \quad \text{and}\quad \hat{q}(X,Y) = tr(\hat{\nu}\mapsto \hat{\eta}(X)d_{Y}\hat{\nu})\]
are equal. 
\end{proof}

As a corollary to Proposition~\ref{prop:hateta} we have the following theorem:

\begin{theorem}[{\cite[Theorem 4.3.7, Proposition 4.3.8]{C2012i}}]
$\hat{f}$ is a Lie-applicable surface and $f$ is a Darboux transform of $\hat{f}$ with parameter $m$. 
\end{theorem}

An obvious question to ask is what happens if we use a different gauge $\tilde{\eta}=\eta -d\tau$ to compute our Darboux transforms. However, by Lemma~\ref{lem:flatgorb}, $\exp(m\tau)\hat{s}\le \hat{f}$ is a parallel subbundle of $d+m\tilde{\eta}$. Hence, we obtain the same Darboux transforms. 

Now $\exp(\wedge^{2}f)$ acts transitively on $\hat{f}\backslash s_{0}$ and, analogously, $\exp(\wedge^{2}\hat{f})$ acts transitively on $f\backslash s_{0}$. Thus, given $s'\le f$ and $\hat{s}'\le \hat{f}$ such that $s'\cap s_{0} = \{0\} = \hat{s}'\cap s_{0}$, there exists $\tau\in \Gamma (\wedge^{2}f)$ and $\hat{\tau}\in \Gamma(\wedge^{2}\hat{f})$ such that 
\[ s' = \exp(m\hat{\tau})s \quad \text{and}\quad \hat{s}' = \exp(m\tau)\hat{s}.\]
By letting $\eta':= \eta - d\tau$ and $\hat{\eta}':=\hat{\eta} -d\hat{\tau}$ we have that $s$ is a parallel subbundle of $d+m\hat{\eta}'$ and $\hat{s}$ is a parallel subbundle of $d+m\eta'$. We therefore have the following proposition:

\begin{proposition}
Suppose that $\hat{f}$ is a Darboux transform of $f$ with parameter $m$ and let $l$ be any rank 2 subbundle of $f+\hat{f}$ with $l\cap s_{0}=\{0\}$. Then there exist gauge potentials $\eta\in\Omega^{1}(f\wedge f^{\perp})$ and $\hat{\eta}\in\Omega^{1}(\hat{f}\wedge \hat{f}^{\perp})$ such that $s:=f\cap l$ is a parallel subbundle of $d+m\hat{\eta}$ and $\hat{s}:=\hat{f}\cap l$ is a parallel subbundle of $d+m\eta$. 
\end{proposition}

\subsection{The enveloping sphere congruence}
In this subsection we will show that the nature of the enveloping sphere congruence $s_{0}$ determines when umbilics appear on a Darboux transform. Furthermore, we will see how we can determine the middle potential of a Darboux transform. 

Suppose that $f$ is an umbilic-free Lie applicable surface and as in Subsection~\ref{subsec:inv} we make the assumption that the signature of $q$ is constant over $\Sigma$. 

\begin{proposition}
$p\in \Sigma$ is an umbilic point of $\hat{f}$ if and only if $s_{0}$ coincides with one of the isothermic sphere congruences at $p$. 
\end{proposition}
\begin{proof}
Suppose that $s_{0}$ coincides with an isothermic sphere congruence $s\le f$ at $p$. Let $\hat{\sigma}\in\Gamma \hat{f}$ be a parallel section of $d+m\eta$, where $\eta\in\Omega^{1}(s\wedge f^{\perp})$ is the isothermic gauge potential associated to $s$. Since $s_{0}$ coincides with $s$ at $p$, we have that $\hat{\sigma}(p)\in s(p)^{\perp}$. Then 
\[ (d\hat{\sigma})_{p} = -m\eta_{p}\hat{\sigma}(p) \in  T_{p}\Sigma\otimes s_{0}(p).\]
Therefore, $p$ is an umbilic point of $\hat{f}$. 

Conversely, suppose that $p$ is an umbilic point of $\hat{f}$. Then there exists $\hat{s}\le \hat{f}$ such that $(d\hat{\sigma})_{p}\in T_{p}\Sigma\otimes \hat{f}(p)$ for all $\hat{\sigma}\in\Gamma \hat{s}$. Since we assumed that $s_{0}$ is never a curvature sphere, we have that $\hat{s}\cap s_{0}=\{0\}$. Now we may choose $\eta\in\Omega^{1}(f\wedge f^{\perp})$ such that $\hat{s}$ is a parallel subbundle of $d+m\eta$. Let $\hat{\sigma}\in\Gamma \hat{s}$ be a parallel section of $d+m\eta$. Then at $p$ 
\[ m\eta_{p}\hat{\sigma}(p) = -(d\hat{\sigma})_{p} \in T_{p}\Sigma\otimes \hat{f}(p). \]
Moreover, since $\eta\in\Omega^{1}(f\wedge f^{\perp})$, $\eta_{p}\hat{\sigma}(p)$ takes values in $f(p)^{\perp}$. Thus, $\eta_{p}\hat{\sigma}(p)$ takes values in $s_{0} = \hat{f}\cap f^{\perp}$. Now for some complementary sphere congruence $s\le f$ to $s_{0}$, we may write
\[ \eta = \sigma_{0}\wedge \omega_{0} + \sigma\wedge \omega,\]
where $\omega_{0},\omega\in\Omega^{1}(f^{\perp})$, $\sigma_{0}\in\Gamma s_{0}$ and $\sigma\in\Gamma s$. Thus 
\[ \eta_{p}\hat{\sigma}(p) = (\sigma(p),\hat{\sigma}(p))\omega_{p} \, \bmod\, T_{p}^{*}\Sigma\otimes f(p).\]
Since $s$ is complementary to $s_{0}$, we must have that $(\sigma(p),\hat{\sigma}(p))$ is non-zero and thus $\omega_{p}\in T_{p}\Sigma\otimes f(p)$. Therefore, $\eta_{p}\in T_{p}\Sigma\otimes (s_{0}(p)\wedge f(p)^{\perp})$. Hence, by Lemma~\ref{lem:pointiso}, $s_{0}$ coincides with an isothermic sphere congruence at $p$. 
\end{proof}

Recall in Subsection~\ref{subsubsec:Dq} that we defined $\Delta_{q}$ and $\zeta_{q}$ associated to a Lie applicable surface. Using Proposition~\ref{prop:isozeta} we obtain the following corollary:

\begin{corollary}
\label{cor:darbumb}
$p$ is an umbilic point of $\hat{f}$ if and only if $\zeta_{q}(s_{0}(p),s_{0}(p))=0$. 
\end{corollary}

Now suppose that $\hat{f}$ is umbilic-free. Then by Corollary~\ref{cor:darbumb}, $\zeta_{q}(s_{0},s_{0})$ is nowhere zero, i.e., $(\Delta_{q}\sigma_{0})\cap s_{0}^{\perp}=\{0\}$ for any lift $\sigma_{0}$ of $s_{0}$. We may then define a rank 4 subbundle of $\underline{\mathbb{R}}^{4,2}$ with signature $(3,1)$, 
\[ V_{q} := s_{0}\oplus d\sigma_{0}(T\Sigma) \oplus \langle \Delta_{q}\sigma_{0}\rangle.\] 
Recall in Definition~\ref{def:envpoint} that we defined the enveloping point $s_{\infty}$ in the plane $f+\hat{f}$ of two Ribaucour transforms as the unique point map in $f+\hat{f}$ satisfying $s_{\infty}^{(1)}\le f+\hat{f}$. Taking lines between the corresponding curvature spheres of $f$ and $\hat{f}$, we obtain $s_{\infty}$ as the intersection of these two lines. 

\begin{proposition}
\label{prop:darbmid}
Let $\eta^{mid}$ denote the middle potential of $f$ and $\hat{\eta}^{mid}$ the middle potential of $\hat{f}$. Then
$V_{q}^{\perp} = s\oplus\hat{s}$ where $s\le f$ is a parallel subbundle of $d+m\hat{\eta}^{mid}$ and $\hat{s}\le \hat{f}$ is a parallel subbundle of $d+m\eta^{mid}$. Furthermore, $s_{\infty}\le V_{q}^{\perp}$.
\end{proposition}

To prove Proposition~\ref{prop:darbmid} we shall use the following lemma:

\begin{lemma}
\label{lem:darbmid}
Suppose that $\eta^{mid}$ is the middle potential. Let $\tau\in \Gamma(\wedge^{2}f)$. Then, 
\[ \eta^{mid}(X) d_{X}\tau -\epsilon^{2} \eta^{mid}(Y)d_{Y}\tau =-\epsilon^{2}\tau,\]
where $X\in\Gamma T_{1}$, $Y\in\Gamma T_{2}$ such that $q_{1}(X,X)=1$ and $q_{2}(Y,Y)=1$. 
\end{lemma}
\begin{proof}
Let $\sigma_{1}\in\Gamma s_{1}$, $\sigma_{2}\in\Gamma s_{2}$ be the special lifts of the curvature spheres such that 
\[ q = -\epsilon^{2}(d\sigma_{2},d\sigma_{2}) + (d\sigma_{1},d\sigma_{1})\]
and let $\tau = \sigma_{1}\wedge\sigma_{2}\in\Gamma (\wedge^{2}f)$. Recall from~(\ref{eqn:midform}) that the middle potential is given by 
\[ \eta^{mid} = \sigma_{1}\wedge \star d\sigma_{1} + \epsilon^{2}\sigma_{2}\wedge \star d\sigma_{2}.\]
Thus, for $v\in\Gamma \underline{\mathbb{R}}^{4,2}$,
\begin{align*}
(\eta^{mid}(X)d_{X}\tau) v&= (\epsilon^{2}\sigma_{2}\wedge d_{X}\sigma_{2})(\sigma_{1}\wedge d_{X}\sigma_{2})v \\
&=-\epsilon^{2}(d_{X}\sigma_{2},d_{X}\sigma_{2})(\sigma_{1},v) \sigma_{2}\\
&= -\epsilon^{2} (\sigma_{1},v) \sigma_{2}. 
\end{align*}
Similarly, $(\eta^{mid}(Y)d_{Y}\tau) v = - (\sigma_{2},v)\sigma_{1}$. Hence, 
\[  (\eta^{mid}(X) d_{X}\tau -\epsilon^{2} \eta^{mid}(Y)d_{Y}\tau)v = -\epsilon^{2}(\sigma_{1},v)\sigma_{2} +\epsilon^{2}(\sigma_{2},v)\sigma_{1} = -\epsilon^{2}\tau v\]
and the result follows. 
\end{proof}

\begin{proof}[Proof of Proposition~\ref{prop:darbmid}]
Let $\hat{\sigma}\in\Gamma \hat{s}$ be a parallel section of $d+m\eta^{mid}$ and let $\sigma_{0}\in\Gamma s_{0}$. Then,  
\begin{align*}
(\Delta_{q}\sigma_{0}, \hat{\sigma}) &= (d_{X}d_{X}\sigma_{0} -\epsilon^{2} d_{Y}d_{Y}\sigma_{0}, \hat{\sigma})\\
&= -(d_{X}\sigma_{0},d_{X}\hat{\sigma}) +\epsilon^{2}(d_{Y}\sigma_{0},d_{Y}\hat{\sigma})\\
&= m (( d_{X}\sigma_{0},\eta^{mid}(X)\hat{\sigma}) -\epsilon^{2} (d_{Y}\sigma_{0},\eta^{mid}(Y)\hat{\sigma}) )\\
&= -m (\eta^{mid}(X)d_{X}\sigma_{0} -\epsilon^{2} \eta^{mid}(Y)d_{Y}\sigma_{0},\hat{\sigma}).
\end{align*}
Now, there exists $\tau\in\Gamma(\wedge^{2}f)$ such that $\sigma_{0} = \tau\hat{\sigma}$. Hence, 
\[ (\Delta_{q}\sigma_{0}, \hat{\sigma}) = -m ( (\eta^{mid}(X) d_{X}\tau -\epsilon^{2} \eta^{mid}(Y)d_{Y}\tau)\hat{\sigma}, \hat{\sigma}) = m\epsilon^{2} (\tau \hat{\sigma},\hat{\sigma}),\]
by Lemma~\ref{lem:darbmid}. By the skew-symmetry of $\tau$, $(\Delta_{q}\sigma_{0}, \hat{\sigma})$ vanishes. 

By a symmetric argument, $(\Delta_{q}\sigma_{0}, \sigma)$ vanishes, where $\sigma$ is a parallel section of $d+m\hat{\eta}^{mid}$.  

Now let $\sigma_{\infty}\in\Gamma s_{\infty}$. Then $d\sigma_{\infty}\in\Omega^{1}(f+\hat{f})$. Then using that $d\sigma_{0}\in\Omega^{1}((f+\hat{f})^{\perp})$ for any $\sigma_{0}\in\Gamma s_{0}$, it is clear that $s_{\infty} \le (s_{0}\oplus d\sigma_{0}(T\Sigma))^{\perp}$, from which it follows that $(\Delta_{q}\sigma_{0}, \sigma_{\infty})$ vanishes.
\end{proof}

As a corollary to Proposition~\ref{prop:darbmid} we obtain the following theorem that tells us how to determine the middle potential of $\hat{f}$:

\begin{theorem}
Suppose that $f$ and $\hat{f}$ are umbilic-free Darboux transforms of each other with parameter $m$. Let $\hat{s}\le \hat{f}$ be the parallel subbundle of $d+m\eta^{mid}$. Then
\begin{equation*} 
d+t\hat{\eta}^{mid} = \Gamma^{\hat{s}}_{s}(1-t/m)\cdot (d+t\eta^{mid}),
\end{equation*}
for $s:= f\cap l$, where $l$ is the line spanned by $\hat{s}$ and $s_{\infty}$.
\end{theorem}

\subsection{Isothermic sphere congruences}
Let $f$ and $\hat{f}$ be umbilic-free Darboux transforms of each other with parameter $m$ and suppose that we are working with the isothermic potential $\eta^{+}$ associated to the isothermic sphere congruence $s^{+}$. Then if $\hat{s}\le \hat{f}$ is the parallel subbundle of $d+m\eta^{+}$ then by Proposition~\ref{prop:hateta}, $\hat{\eta}$ defined by 
\[ d+t\hat{\eta} = \Gamma^{\hat{s}}_{s^{+}}(1-t/m)\cdot (d+t\eta^{+})\]
is a closed 1-form. Recall that we split $\eta=\eta_{0}+\eta_{s^{+}}$ and $\hat{\eta}=\hat{\eta}_{0}+\hat{\eta}_{\hat{s}}$, where $\eta_{0},\hat{\eta}_{0}\in \Omega^{1}(s_{0}\wedge l^{\perp})$, $\eta_{s^{+}}\in \Omega^{1}(s^{+}\wedge l^{\perp})$ and $\hat{\eta}_{\hat{s}}\in \Omega^{1}(\hat{s}\wedge l^{\perp})$. Now, in the proof of Proposition~\ref{prop:hateta} we saw that $\eta_{0} = \hat{\eta}_{0}$ and since we are working with the isothermic potential $\eta^{+}\in \Omega^{1}(s^{+}\wedge f^{\perp})$, we have that $\eta_{0}=0$. Thus, $\hat{\eta}\in \Omega^{1}(\hat{s}\wedge l^{\perp})$. Hence, $\hat{s}$ is isothermic and we shall denote it $\hat{s}^{+}$. A symmetric argument yields an analogous result for $s^{-}$.

\begin{proposition}
\label{prop:darbiso}
We may label the isothermic sphere congruences $\hat{s}^{+}$ and $\hat{s}^{-}$ of $\hat{f}$ such that $\hat{s}^{\pm}$ is a parallel subbundle of $d+m\eta^{\pm}$. Furthermore
\[ d+t\hat{\eta}^{\pm} = \Gamma^{\hat{s}^{\pm}}_{s^{\pm}}(1-t/m)\cdot (d+t\eta^{\pm}).\]
\end{proposition}

Proposition~\ref{prop:darbiso} shows that Darboux transforms of Lie applicable surfaces are induced by the Darboux transforms of their isothermic sphere congruences~\cite{BDPP2011,H2003}. On the other hand, given a Darboux transform $\hat{s}^{+}$ of one of the isothermic sphere congruence, say $s^{+}$, we have that $\hat{f}:= s_{0}\oplus \hat{s}^{+}$, where $s_{0}=f\cap (\hat{s}^{+})^{\perp}$, is a Darboux transform of $f$. This is our justification for using the term ``Darboux transform" instead of ``B\"acklund transform".

We now give a result concerning the lines joining ``opposite" isothermic sphere congruences:
\begin{proposition}
Let $l_{1}= s^{+}\oplus \hat{s}^{-}$ and $l_{2}= s^{-}\oplus \hat{s}^{+}$. Then $l_{1}\cap l_{2}=s_{\infty}$. 
\end{proposition}
\begin{proof}
By Proposition~\ref{prop:chrlifts}, $\eta^{-}=\sigma^{-}\wedge d\sigma^{+}$, where $\sigma^{\pm}$ are Christoffel dual lifts of $s^{\pm}$. By Proposition~\ref{prop:darbiso}, there exists $\hat{\sigma}\in\Gamma \hat{s}^{-}$ such that $\hat{\sigma}$ is a parallel section of $d+m\eta^{-}$. Thus, 
\[ d\hat{\sigma}= -m(\sigma^{-},\hat{\sigma})d\sigma^{+}\, \bmod\, \Omega^{1}(f).\] 
Hence, $\sigma_{\infty}:= \hat{\sigma} + m(\sigma^{-},\hat{\sigma})\sigma^{+}\in \Gamma l_{1}$ and satisfies $d\sigma_{\infty}\in\Omega^{1}(f+\hat{f})$. Since $s_{\infty}$  is the unique point in $f+\hat{f}$ satisfying $s_{\infty}^{(1)}\le f+\hat{f}$, we have that $\sigma_{\infty}\in \Gamma s_{\infty}$. Therefore, $s_{\infty}\le l_{1}$. Similarly, $s_{\infty}\le l_{2}$ and the result follows. 
\end{proof}

\section{Associate surfaces}
\label{sec:asssurf}
Let us recall the definition of $O$-surfaces given in~\cite{KS2003}: suppose that $x^{1},...,x^{n}:\Sigma\to\mathbb{R}^{3}$ are Combescure transformations\footnote{That is, the curvature directions of $x^{i}$ are parallel to the curvature directions of $x^{j}$ for all $i,j\in\{1,...,n\}$.} of each other and let the subbundles $T_{1},T_{2}\le T\Sigma$ denote the induced curvature subbundles on $T\Sigma$. Let $\kappa_{1}^{i}$ and $\kappa_{2}^{i}$ denote the principal curvatures of $x^{i}$ along $T_{1}$ and $T_{2}$, respectively, and define row vectors
\[ K_{j}:=(1/\kappa_{j}^{1},...,1/\kappa_{j}^{n}),\]
for $j\in\{1,2\}$. Then we say that $\{x^{1},...,x^{n}\}$ is a system of $O$-surfaces if there exists a constant symmetric $n\times n$ matrix $S$ such that 
\[ K_{1}SK^{t}_{2}=0.\]
In this section we shall see how a system of $O$-surfaces arises from an $\Omega$-surface. 

In \cite{D1911iii}, Demoulin defines an associate surface of an umbilic-free $\Omega$-surface: suppose that $x:\Sigma\to\mathbb{R}^{3}$ is an $\Omega$-surface and in terms of curvature line coordinates $(u,v)$ the third fundamental form of $x$ is given by $I\!I\!I=p^{2} du^{2} + r^{2} dv^{2}$. Then there exists a Combescure transformation $x^{D}:\Sigma\to \mathbb{R}^{3}$ of $x$ and there exist functions $U$ of $u$ and $V$ of $v$ such that
\begin{eqnarray}
\label{eqn:assocsurf}
\left(\frac{1}{\kappa_{1}} - \frac{1}{\kappa_{2}}\right)\left(\frac{1}{\kappa^{D}_{1}} - \frac{1}{\kappa^{D}_{2}}\right) = 
-\epsilon^{2}\frac{U^{2}}{p^{2}} +\frac{V^{2}}{r^{2}},
\end{eqnarray}
where $\kappa_{1}$ and $\kappa_{2}$ denote the principal curvatures of $x$, $\kappa^{D}_{1}$ and $\kappa^{D}_{2}$ denote the principal curvatures of $x^{D}$ and $\epsilon\in\{1,i\}$. Conversely, if two surfaces are in such a relation then they are $\Omega$-surfaces. 

Suppose that $f:\Sigma \to \mathcal{Z}$ is an umbilic-free $\Omega$-surface. Then there exists a closed 1-form $\eta\in\Omega^{1}(f\wedge f^{\perp})$ such that the quadratic differential associated to $\eta$ is non-degenerate. Let $\mathfrak{q}_{\infty}$ and $\mathfrak{p}$ be a space form vector and point sphere complex with $|\mathfrak{q}_{\infty}|^{2}=0$ and $|\mathfrak{p}|^{2}=-1$, i.e., 
\[ \mathfrak{Q}^{3}:=\{y\in\mathcal{L}:(y,\mathfrak{q}_{\infty})=-1, (y,\mathfrak{p})=0\}\]
has sectional curvature $\kappa=0$ and $\mathfrak{Q}^{3}\cong \mathbb{R}^{3}$. Then we may choose a null vector $\mathfrak{q}_{0}\in \langle \mathfrak{p}\rangle^{\perp}$ such that $(\mathfrak{q}_{0},\mathfrak{q}_{\infty})=-1$. Thus $\langle \mathfrak{q}_{\infty},\mathfrak{p},\mathfrak{q}_{0}\rangle^{\perp}\cong \mathbb{R}^{3}$ and we have an isometry 
\[\phi:\langle \mathfrak{q}_{\infty},\mathfrak{p},\mathfrak{q}_{0}\rangle^{\perp}\to \mathfrak{Q}^{3},\quad x\mapsto x + \mathfrak{q}_{0} + \frac{1}{2}(x,x)\mathfrak{q}_{\infty}.\]
We can use this to identify $\mathfrak{f}:= f\cap\mathfrak{Q}^{3}$ with $x:\Sigma\to\mathbb{R}^{3}$. We then have that $d\mathfrak{f} = dx + (dx,x)\mathfrak{q}_{\infty}$ and $\mathfrak{t} = n + (n,x)\mathfrak{q}_{\infty} + \mathfrak{p}$. 

Now $(\eta \mathfrak{p},\mathfrak{q}_{\infty})$ is a closed 1-form, so there exists (up to addition of a constant) $\lambda:\Sigma\to \mathbb{R}$ such that $d\lambda = (\eta \mathfrak{p},\mathfrak{q}_{\infty})$. Then we may gauge $\eta$ by $\tau:= - \lambda\mathfrak{f}\wedge \mathfrak{t}$ to obtain $\tilde{\eta}:=\eta-d\tau$ with $(\tilde{\eta}\mathfrak{p},\mathfrak{q}_{\infty})=0$. Therefore, we shall assume that $(\eta\mathfrak{p},\mathfrak{q}_{\infty})=0$. From this we can deduce that $\eta$ is of the form 
\[ \eta = \mathfrak{f} \wedge d\mathfrak{f}\circ A + \mathfrak{t}\wedge d\mathfrak{t}\circ B,\]
for some $A,B\in \Gamma End(T\Sigma)$. The closure of $\eta$ implies that $\eta\mathfrak{q}_{\infty} = -d\mathfrak{f}\circ A$ and $\eta\mathfrak{p}=-d\mathfrak{t}\circ B$ are closed and that
\begin{eqnarray}
\label{eqn:assurf} d\mathfrak{f}\curlywedge d\mathfrak{f}\circ A + d\mathfrak{t}\curlywedge d\mathfrak{t}\circ B=0.
\end{eqnarray}
The closure of $d\mathfrak{f}\circ A$ implies that $dx\circ A$ is closed. Furthermore, by Lemma~\ref{prop:mceqn}, we have that $\eta(T_{i})\le f\wedge f_{i}$ and thus 
\[ A\in \Gamma(T_{1}^{*}\otimes T_{1} \oplus T_{2}^{*}\otimes T_{2}).\]
Therefore, locally there exists $x^{D}:\Sigma\to \mathbb{R}^{3}$ such that $dx^{D} = dx\circ A$ and $x^{D}$ has parallel curvature directions to $x$. Similarly there exists $\hat{x}:\Sigma\to \mathbb{R}^{3}$ such that $d\hat{x} = dn\circ B$ and $B\in\Gamma(T_{1}^{*}\otimes T_{1} \oplus T_{2}^{*}\otimes T_{2})$. Thus, $\hat{x}$ also has parallel curvature directions to $x$. From Equation~(\ref{eqn:assurf}) and Rodrigues' equations, we can then deduce that 
\begin{eqnarray}
\label{eqn:asscurv}
\frac{1}{\kappa_{1}\kappa^{D}_{2}} +  \frac{1}{\kappa_{2}\kappa^{D}_{1}} -  \frac{1}{\hat{\kappa}_{1}} - \frac{1}{\hat{\kappa}_{2}}=0.
\end{eqnarray}

Conversely, given Combescure transformations $x^{D}$ and $\hat{x}$ of $x$ such that~(\ref{eqn:asscurv}) is satisfied we may define a closed 1-form
\[ \eta = \mathfrak{f}\wedge (dx^{D} + (dx^{D},x)\mathfrak{q}_{\infty}) + \mathfrak{t}\wedge (d\hat{x} + (d\hat{x}, x)\mathfrak{q}_{\infty}).\]
Hence, we have arrived at the following result:

\begin{theorem}
\label{thm:asssurf}
An umbilic-free surface $x:\Sigma\to\mathbb{R}^{3}$ is an $\Omega$-surface if and only if there exists an associate surface $x^{D}:\Sigma\to\mathbb{R}^{3}$ and an associate Gauss map $\hat{x}:\Sigma\to\mathbb{R}^{3}$ that are Combescure transformations of $x$ such that the principal curvatures of $x$, $x^{D}$ and $\hat{x}$ satisfy (\ref{eqn:asscurv}).
\end{theorem}

\begin{remark}
We shall assume that $x^{D}$ and $\hat{x}$ are oriented so that the Gauss map of these surfaces is $-n$.
\end{remark}

\begin{remark}
The addition of a constant $c$ to $\lambda$ sends 
\[ x^{D}\mapsto x^{D} +cn \quad \text{and}\quad \hat{x}\mapsto\hat{x}-cx. \]
Thus we get a parallel surface to $x^{D}$. The behaviour of $\hat{x}$ under this change is our motivation for calling $\hat{x}$ an associate Gauss map. 
\end{remark}
By letting 
\[ S = \begin{pmatrix}
0 & 1 & 0 & 0\\
1 & 0 & 0 & 0\\
0 & 0 & 0 & 1\\
0 & 0 & 1 & 0
\end{pmatrix}
\]
one can see that condition~(\ref{eqn:asscurv}) shows that $\{x, x^{D},\hat{x},n\}$ is a system of $O$-surfaces, where we consider the Gauss map $n$ to be oriented so that its principal curvatures are both $-1$.

We have that the quadratic form of $\eta$ is given by
\[ q = -(dx, dx^{D}) - (dn,d\hat{x}).\]
On the other hand, in terms of curvature line coordinates $(u,v)$, we have that 
\[ q = -\epsilon^{2}U^{2}du^{2} + V^{2}dv^{2},\]
for some functions $U$ of $u$ and $V$ of $v$. Hence,
\[ - \epsilon^{2} U^{2} = \left(\frac{1}{\kappa_{1}\kappa^{D}_{1}} - \frac{1}{\hat{\kappa}_{1}}\right)p^{2}\quad \text{and} \quad V^{2} =  \left(\frac{1}{\kappa_{2}\kappa^{D}_{2}} - \frac{1}{\hat{\kappa}_{2}}\right)r^{2}\]
and thus 
\[ \frac{1}{\hat{\kappa}_{1}} = \frac{\epsilon^{2}U^{2}}{p^{2}} + \frac{1}{\kappa_{1}\kappa^{D}_{1}}\quad \text{and} \quad \frac{1}{\hat{\kappa}_{2}} = -\frac{V^{2}}{r^{2}} + \frac{1}{\kappa_{2}\kappa^{D}_{2}}.\]
Then substituting this into~(\ref{eqn:asscurv}) yields~(\ref{eqn:assocsurf}). Hence, $x^{D}$ is an associate surface of $x$, in the sense of~\cite{D1911iii}.

We may write $\eta$ as 
\[ \eta = Ad_{\exp(x\wedge \mathfrak{q}_{\infty})}(\mathfrak{q}_{0}\wedge dx^{D}+ \xi\wedge d\hat{x}),\]
where $\xi:= n+\mathfrak{p}$. By the symmetry of Equation~(\ref{eqn:asscurv}), $x^{D}$ is an $\Omega$-surface with closed 1-form 
\[ \eta^{D} :=  Ad_{\exp(x^{D}\wedge \mathfrak{q}_{\infty})}(\mathfrak{q}_{0}\wedge dx - \xi^{D}\wedge d\hat{x}),\]
where $\xi^{D}:=-n+\mathfrak{p}$. Furthermore, the quadratic differential $q^{D}$ defined by $\eta^{D}$ agrees with $q$. 

\begin{theorem}
An associate surface of an $\Omega$-surface is itself an $\Omega$-surface.
\end{theorem}

\appendix
\section{The quadratic differential}
In this appendix we prove some facts about the quadratic differential that arises in the definition of Lie applicability (see Definition~\ref{def:lieapp}) in order to prove Theorem~\ref{thm:invariant}. In particular, we shall characterise the condition that a quadratic differential is divergence-free in terms of certain special lifts of the curvature spheres. 

Suppose that $f:\Sigma\to \mathcal{Z}$ is an umbilic-free Legendre immersion and let $q\in \Gamma ((T_{1}^{*})^{2}\oplus (T_{2}^{*})^{2})$ be a quadratic differential. Then for any representative metric $g$ of the conformal structure $c$, there exists a symmetric trace free endomorphism $Q\in \Gamma\textrm{End}(T\Sigma)$ such that 
\[ q(X,Y) = g(X, Q(Y)),\]
for any $X,Y\in \Gamma T\Sigma$.

Now the conformal structure $c$ gives rise to a product structure $J$ which acts as $id$ on $T_{1}$ and $-id$ on $T_{2}$. Since the Hodge star operator $\star$ induced by $c$ acts as $id$ on $T^{*}_{1}$ and $-id$ on $T^{*}_{2}$ we have, for any $\alpha\in \Gamma\textrm{End}(T\Sigma)$, 
\[ \star \alpha = \alpha\circ J.\]
One then deduces the following lemma:

\begin{lemma}
\label{lem:hodgeQ}
$Q\in\Gamma End(T\Sigma)$ is trace-free and symmetric with respect to $c$ if and only if $\star Q = - J \circ Q$.
\end{lemma}

\begin{corollary}
\label{cor:divergence}
Suppose that $Q\in\Gamma \textrm{End}(T\Sigma)$ is trace-free and symmetric with respect to $c$. Let $g$ be a representative metric for $c$ with induced Levi-Civita connection $\nabla$. Then $d^{\nabla}\!\star Q=0$, i.e., $Q$ is divergence-free, if and only if $d^{\nabla}Q=0$. 
\end{corollary}
\begin{proof}
Since $\nabla$ is the Levi-Civita connection for $g$, we have that $d^{\nabla}J=0$. Then, using Lemma~\ref{lem:hodgeQ} and the Leibniz rule, 
\[ d^{\nabla}\star Q = - (d^{\nabla}J)\circ Q - J\circ d^{\nabla}Q = -J\circ d^{\nabla}Q,\]
and the result follows.
\end{proof}

We say that $q$ is \textit{divergence-free with respect to $c$} if for any representative metric $g$, the endomorphism $Q\in \Gamma \textrm{End}(T\Sigma)$ defined by
\[ q(X,Y)= g(X,Q(Y))\]
is divergence-free.  

Assume that the signature of $q$ is constant over $\Sigma$. Recall from Subsection~\ref{subsec:inv} that after possibly multiplying $q$ by $\pm 1$ and reordering the curvature sphere congruences $s_{1}$ and $s_{2}$, there exists (unique up to sign) lifts $\sigma_{1}\in\Gamma s_{1}$ and $\sigma_{2}\in \Gamma s_{2}$ such that 
\[ q = -\epsilon^{2}(d\sigma_{2},d\sigma_{2}) + (d\sigma_{1},d\sigma_{1}),\]
where $\epsilon\in\{0,1,i\}$. 

\begin{lemma}
\label{app:quad}
$q$ is divergence-free with respect to the conformal structure $c$ on $T\Sigma$ if and only if $d_{1}\sigma_{1}\in\Gamma T_{1}^{*}\otimes s_{2}$ and $\epsilon^{2}d_{2}\sigma_{2}\in\Gamma T_{2}^{*}\otimes s_{1}$.
\end{lemma}
\begin{proof}
Let $g$ be a representative metric of the conformal structure $c$ and let $\nabla$ denote the Levi-Civita connection of $g$. Since $T_{1},T_{2}$ are the maximally isotropic subbundles of this metric, we have that $\nabla_{Z}X\in\Gamma T_{1}$ and $\nabla_{Z}Y\in\Gamma T_{2}$ for any $X\in\Gamma T_{1}$, $Y\in \Gamma T_{2}$ and $Z\in \Gamma T\Sigma$. 

Let $Q\in \Gamma \textrm{End}(T\Sigma)$ such that 
\[ q(X,Y) = g(X, Q(Y)).\]
Since $q\in \Gamma ( (T_{1}^{*})^{2}\oplus (T_{2}^{*})^{2})$, we have that $Q( T_{1})\le T_{2}$ and $Q (T_{2})\le T_{1}$. Hence, $Q$ is symmetric and trace-free with respect to $g$. Now for $X\in\Gamma T_{1}$ and $Y\in\Gamma T_{2}$
\begin{align*}
d_{Y}(q(X,X)) &= - \epsilon^{2}d_{Y}(d_{X}\sigma_{2},d_{X}\sigma_{2})\\
&= - 2\epsilon^{2}(d_{Y}d_{X}\sigma_{2},d_{X}\sigma_{2})\\
&= -2\epsilon^{2}((d_{X}d_{Y}\sigma_{2},d_{X}\sigma_{2}) + (d_{[Y,X]}\sigma_{2},d_{X}\sigma_{2})).\\
\end{align*}
 On the other hand, since $\nabla$ is the Levi-Civita connection we have that
\[ d_{Y}(q(X,X)) = d_{Y}(g(X,Q(X))) = g(\nabla_{Y}X,Q(X)) + g(X, \nabla_{Y}Q(X)).\]
Furthermore, $-\epsilon^{2}(d_{[Y,X]}\sigma_{2},d_{X}\sigma_{2})$ is equal to 
\[ q([Y,X],X) = g(Q(\nabla_{Y}X - \nabla_{X}Y),X) = g(Q(\nabla_{Y}X),X),\]
since $Q(\nabla_{X}Y)\in\Gamma T_{1}$. Hence, 
\[
-2\epsilon^{2} (d_{X}d_{Y}\sigma_{2},d_{X}\sigma_{2}) = g(\nabla_{Y}Q(X)- Q(\nabla_{Y}X),X).
\]
Therefore, since $\nabla_{Y}Q(X)- Q(\nabla_{Y}X)\in \Gamma T_{2}$, $-2\epsilon^{2}(d_{X}d_{Y}\sigma_{2},d_{X}\sigma_{2})=0$ if and only if $(\nabla_{Y}Q)(X)= 0$. One can then check that $-2\epsilon^{2}(d_{X}d_{Y}\sigma_{2},d_{X}\sigma_{2})=0$ if and only if $\epsilon^{2}d_{Y}\sigma_{2}\in \Gamma T_{2}^{*}\otimes s_{1}$. Similarly, one can show that $(\nabla_{X}Q)(Y) = 0$ if and only if $d_{X}\sigma_{1}\in\Gamma T_{1}^{*}\otimes s_{2}$. Therefore, 
\[ (d^{\nabla}Q)(X,Y) = (\nabla_{X}Q)(Y) - (\nabla_{Y}Q)(X) = 0\]
if and only if $\epsilon^{2}d_{2}\sigma_{2}\in\Gamma T_{2}^{*}\otimes s_{1}$ and $d_{1}\sigma_{1}\in\Gamma T_{1}^{*}\otimes s_{2}$. The result follows by applying Corollary~\ref{cor:divergence}.
\end{proof}

\bibliographystyle{plain}
\bibliography{bibliography2015}

\end{document}